\documentclass[10pt, reqno]{amsart}
\usepackage{graphicx, amssymb, amsmath, amsthm}
\numberwithin{equation}{section}

\usepackage{color}

\usepackage[backref=page]{hyperref}  

\let\Re=\undefined\DeclareMathOperator*{\Re}{Re}
\let\Im=\undefined\DeclareMathOperator*{\Im}{Im}

\newcommand{\R}{\mathbb{R}}
\newcommand{\C}{\mathbb{C}}

\newcommand{\Z}{\mathbb{Z}}

\newcommand\LL{\mathcal{L}}
\newcommand\A{\bf A}

\newtheorem{theorem}{Theorem}[section]

\newtheorem{lemma}[theorem]{Lemma}

\newtheorem{proposition}[theorem]{Proposition}

\theoremstyle{definition}
\newtheorem{definition}[theorem]{Definition}
\newtheorem{remark}[theorem]{Remark}

\makeatletter
\newcommand{\Extend}[5]{\ext@arrow0099{\arrowfill@#1#2#3}{#4}{#5}}
\makeatother

\begin{document}
\title[Two-particles in Aharonov-Bohm
field]{Dispersive estimates for two-particles Schr\"odinger and wave equations in the Aharonov-Bohm
field }

\author{Xiaofen Gao}
\address{Department of Mathematics, Beijing Institute of Technology, Beijing 100081}
\email{gaoxiaofen@bit.edu.cn}

\author{Junyong Zhang}
\address{Department of Mathematics, Beijing Institute of Technology, Beijing 100081}
\email{zhang\_junyong@bit.edu.cn}

\author{Jiqiang Zheng}
\address{Institute of Applied Physics and Computational Mathematics, Beijing 100088}
\email{zhengjiqiang@gmail.com; zheng\_jiqiang@iapcm.ac.cn}

\begin{abstract}
We study the dispersive behaviors of two-particles Schr\"odinger and wave equations in the Aharonov-Bohm
field. In particular, we
 prove  the Strichartz estimates
for Schr\"odinger and wave equations in this setting. The key point is to construct spectral measure of  Schr\"odinger operator
with an Aharonov-Bohm type potential in $\R^4$.
 As applications, we finally prove a scattering theory for  the nonlinear defocusing subcritical two-particles Schr\"odinger equation with Aharonov-Bohm potential.

\end{abstract}

\maketitle

\begin{center}
 \begin{minipage}{120mm}
 \end{minipage}
 \end{center}


\section{Intorduction}
In this paper, we continue our program \cite{CYZ, FZZ, GYZZ} studying the dispersive equations in Aharonov-Bohm field.
 The Aharonov-Bohm magnetic potential reads
\begin{equation}\label{AB}
A_B:\R^2\setminus\{(0,0)\}\to\R^2,
\quad
A_B(x)=\alpha\left(-\frac{x_2}{|x|^2},\frac{x_1}{|x|^2}\right),
\quad
\alpha\in\R,
\quad
x=(x_1,x_2)
\end{equation}
so that the Hamiltonian becomes
\begin{equation}\label{eq:H}
H_A=\left(-i\nabla+\alpha\left(-\frac{x_2}{|x|^2},\frac{x_1}{|x|^2}\right)\right)^2.
\end{equation}
The vector potential $A_B(x)$ generates a $\delta$-type magnetic field
$${ B}=\nabla\times A_B(x)=2\pi \alpha \delta,$$
which was initially studied by Aharonov and Bohm in \cite{AB} to show the significance of electromagnetic vector potentials in quantum mechanics.
The Aharonov-Bohm effect predicted in \cite{AB} is one of the most interesting and intriguing phenomena of quantum physics.
This effect occurs when electrons propagate in a domain with a zero magnetic field but with a nonzero vector potential $A_B$, see \cite{PT89} and the references therein.
In another typical cosmic-string scenarios observed by Alford and Wilczek \cite{AW}, the fermionic charges can be non-integer multiples of the Higgs charges.
As the flux is quantized with respect to the Higgs charge, this will lead to  a non-trivial Aharonov-Bohm scattering of these fermions.
From the mathematical points,
we refer to \cite{FFFP1, FFFP} and references therein for an overview of the spectral theory of this Hamiltonian in Aharonov-Bohm magnetic field. This Hamiltonian model is doubly critical, because of the scaling invariance of the model and the singularities of the potentials, which are not locally integrable.
 In $\R^2$,
a generalization of $H_A$ was considered in \cite{FFFP1, FFFP} in which
the dispersive estimates were proved for Schr\"odinger equation. Recently, in \cite{FZZ,GYZZ}, the authors have
 proved the Strichartz estimates
for wave and Klein-Gordon equations in the Aharonov-Bohm magnetic fields. It worths mentioning that the wave and Klein-Gordon equations are lack of pseudoconformal invariance which plays an important role in \cite{FFFP1} for Schr\"odinger equation, hence we have to use a new method in \cite{FZZ,GYZZ}.
  \vspace{0.2cm}

In the present paper, we consider the Hamiltonian for magnetic many-particles Schr\"odinger operator arising in the Hall effect \cite{J89, L83}. In particular,
we recall  the Hamiltonian for magnetic multi-particle
Dirichlet forms with Aharonov-Bohm type in \cite{HHLT}.
 Let ${\vec x}_j=(x_{j1}, x_{j2})\in\R^2$, $j=1,2,\ldots, N$ and let
\begin{equation}\label{Fj}
{\bf F}_j=\alpha\Big(-\sum_{k\neq j}\frac{{x}_{j2}-{ x}_{k2}}{r_{jk}^2}, \sum_{k\neq j}\frac{{ x}_{j1}-{x}_{k1}}{r_{jk}^2}\Big), \quad
r_{jk}^2=\sum_{i=1}^2 (x_{ji}-x_{ki})^2,
\end{equation}
then the Hamiltonian is given by
\begin{equation}\label{H-N}
H_{A,N}=\sum_{j=1}^N (-i\nabla_j+{\bf F}_j)^2,
\end{equation}
which is singular in the set $$S=\cup_{j\neq k} S_{j,k},$$ where for all $j, k\in\{1,2,\ldots,N\}, j\neq k$,
\begin{equation}\label{sing-set}
S_{j,k}=\{x=(\vec x_1,\vec x_2\ldots, \vec x_N)\in \R^{2N}=\R^{2}\times\R^2\times\ldots\times\R^2:\vec x_j=\vec x_k\}.
\end{equation}
We aim to study the dispersive behaviors of the Schr\"odinger equation
\begin{equation}\label{equ:S}
\begin{cases}
i\partial_t u+H_{A,N} u=0,\qquad &(t,x)\in\R\times\R^{2N}\setminus S,\\
u(0,x)=u_0,\qquad &x\in\R^{2N},
\end{cases}
\end{equation}
and the wave equation
\begin{equation}\label{equ:w}
\begin{cases}
\partial_{tt} u+H_{A,N} u=0,\qquad &(t,x)\in\R\times\R^{2N}\setminus S,\\
u(0,x)=f,\quad \partial_tu(0,x)=g,\qquad &x\in\R^{2N}.
\end{cases}
\end{equation}
The research on the dispersive and Strichartz estimates of the wave and Schr\"odinger propagators with one single potential has a long history, we refer to \cite{RS} and
the survey \cite{S} and the references therein.
From the physical perspective, it is natural to ask whether similar estimates hold in the presence of interaction potentials.
In the $N$-particles interaction case, there is much fewer result than single case. It is known that the local-in-time dispersive and Strichartz estimates
follow from the kernel in  \cite{Fu} with a large class of potential interactions. Recently, Hong \cite{H} has obtained the global-in-time
Strichartz estimates but with ``small" interaction potentials. Chong, Grillakis, Machedon and Zhao \cite{CGMZ} proved the global-in-time Strichartz estimates
for the equation $\big(i\partial_t-\Delta_x-\Delta_y+\frac1{N} V_N(x-y)\big) u(t,x,y)=F$ in mixed coordinates norm $L^q_t(\R;L^r_x L^2_y(\R^{n}\times\R^n))$.
To the best of our knowledge, in the $N$-particles case, there is no result about the decay and Strichartz estimates for \eqref{equ:S} and \eqref{equ:w}.
The  two main obstacles to prove the decay estimates arise from the scaling critical Aharonov-Bohm potential and interactions between the many particles.
Fortunately, in the special two-particles case $N=2$, which is the simplest case, the Hamiltonian $H_{A,2}$ can be reduced to a one-particle Hamiltonian $\LL_{\A}$ in \eqref{LA} below but with more singular
Aharonov-Bohm type potential in higher dimensions. Because of this, we can obtain the decay estimates and Strichartz estimates in
the two-particles case.

\vspace{0.2cm}

Our main results about the Strichartz estimates are as follows.
\begin{theorem}\label{thm:stri-S}
Let $N=2$ and let
$u(t,x)$ be a solution of Schr\"odinger equation \eqref{equ:S}, then there exists a constant $C$ such that
\begin{equation}\label{stri}
\|u(t,x)\|_{L^{q}_{t}(\R;L_x^r(\mathbb{R}^4))}\leq C\|u_0\|_{L^2_x(\mathbb{R}^4)}
\end{equation}
where
\begin{equation}\label{adm}
(q,r)\in\Lambda^S_0:=\Big\{(q,r)\in[2,\infty]\times[2,\infty):\tfrac2q=4\big(\tfrac12-\tfrac1r\big)\Big\}.
\end{equation}

\end{theorem}

\begin{remark} By using dual argument, the endpoint Strichartz estimates \eqref{stri} with $(q,r)=(2,4)$ gives
the inhomogeneous Strichartz estimates
$$\Big\|\int_0^t e^{i(t-s) H_{A,2}}F(s) ds\Big\|_{L^2_t(\R;L^4_x(\R^4))} \leq C\|F\|_{L^2_t(\R;L^{\frac43}_x(\R^4))}.$$
Then by using the argument in \cite[Remark 8.8]{HZ}, we can obtain the uniform resolvent estimate
\begin{equation}
\|(H_{A,2}-z)^{-1}\|_{L^{\frac43}_x(\R^4)\to L^{4}_x(\R^4)}\leq C,
\end{equation}
where the constant $C$ is independent of $z\in\C\setminus\R^+$.
\end{remark}

\begin{remark}
It would be interesting to generalize the result to $N$-particles with $N\geq3$, but the method of this paper is not available.
Compared with the result of \cite{H}, here we consider the scaling critical magnetic interaction potentials and the interaction between potentials may be larger.
In constrast to the Strichartz estimates of \cite{CGMZ}, the dispersive and Strichartz estimates \eqref{stri} are available for the norm of $L^{q}_{t}(\R;L_x^r(\mathbb{R}^4))$ but not  for the mixed norm of $L^{q}_{t}(\R;L^r(\mathbb{R}^2)L^2(\R^2))$.\end{remark}

\begin{theorem}\label{thm:stri-w}
Let $N=2$ and let
$u(t,x)$ be a solution of wave equation \eqref{equ:w}, then there exists a constant $C$ such that
\begin{equation}\label{stri-w}
\|u(t,x)\|_{L^{q}_{t}(\R;L_x^r(\mathbb{R}^4))}\leq C\Big(\|f\|_{\dot{H}^s(\mathbb{R}^4)}+\|g\|_{\dot{H}^{s-1}(\mathbb{R}^4)}\Big)
\end{equation}
where $s\geq0$ and
\begin{equation}\label{adm-w}
(q,r)\in\Lambda^w_s:=\Big\{(q,r)\in[2,\infty]\times[2,\infty):\tfrac2q\leq 3\big(\tfrac12-\tfrac1r\big), s=4\big(\tfrac12-\tfrac1r\big)-\tfrac1q\Big\}.
\end{equation}

\end{theorem}

 By the change of variables
\begin{equation}\label{equ:y12x12cha}
{\vec y}_1=\frac{\vec x_1-\vec x_2}{\sqrt{2}},\quad {\vec y}_2=\frac{\vec x_1+\vec x_2}{\sqrt{2}},
\end{equation}
and denoting $y=(\vec{y}_1,\vec{y}_2)\in\R^4$ and
$$v(\vec{y}_1,\vec{y}_2)=u\big(\tfrac{\vec{y}_1+\vec{y}_2}{\sqrt{2}},-\tfrac{\vec{y}_1-\vec{y}_2}{\sqrt{2}}\big)\quad\text{or}\quad
u(\vec{x}_1,\vec{x}_2)=v(\vec{y}_1,\vec{y}_2),$$
then, we derive
\begin{align}\label{equ:changes-HL}
  H_{A,2} u(\vec x_1, \vec x_2)=&
  \LL_{\A} v(\vec y_1, \vec y_2),
\end{align}
 where
\begin{equation*}
\LL_{{\A}}=\Big(-i\nabla_y+{\A}(y)\Big)^2, \quad {\A}(y)=\alpha\Big(\frac {-y_{12}}{|y_{11}|^2+|y_{12}|^2},\frac {y_{11}}{|y_{11}|^2+|y_{12}|^2},0,0\Big).
\end{equation*}
Hence,
\begin{align}\label{equ:HAB}
  \big\| e^{itH_{A,2}}u_0(x)\big\|_{L^{q}_{t}(\R;L_x^r(\mathbb{R}^4))}\simeq & \big\| e^{it\LL_{{\A}}}v_0(y)\big\|_{L^{q}_{t}(\R;L_y^r(\mathbb{R}^4))}.
\end{align}
Thus, we can reduce to consider one-particle operator $\LL_{\A}$.

As applications, we further
 study the  nonlinear subcritical Schr\"odinger equation
\begin{equation}\label{NLS}
\begin{cases}
i\partial_t u+H_{A,2} u=|u|^{p-1}u,\qquad (t,x)\in\R\times(\R^4\setminus S),\\
u(0,x)=u_0\in H_A^1(\R^4),\qquad x\in\R^4.
\end{cases}
\end{equation}
By changing variables \eqref{equ:y12x12cha}, it is equivalent to consider the nonlinear Schr\"odinger equation
 with Aharonov-Bohm potential in 4D
 \begin{equation}\label{equ:nlsab4}
\begin{cases}
i\partial_t u+\LL_{{\A}} u=|u|^{p-1}u,\qquad (t,x)\in\R\times(\R^4\setminus\{0,0, x_3, x_4\}),\\
u(0,x)=u_0\in H^1_{\bf A}(\R^4),\qquad x\in\R^4,
\end{cases}
\end{equation}
where $u:\;\R\times\R^4\to\C,\;2<p<3,$ and $\LL_{{\A}}=\big(-i\nabla_x+{\A}(x)\big)^2.$\vspace{0.2cm}

 We aim to establish the scattering theory for the nonlinear Schr\"odinger equation \eqref{equ:nlsab4}. In particular, when ${\A}=0$, there are rich results about scattering theory for the nonlinear Schr\"odinger equation. For the defocusing and subcritical nonlinear Schr\"odinger cases, Ginibre-Velo \cite{GV} proved the scattering theory  in  energy space $H^1(\R^n)$ with $n\geq3$ by using the classical Morawetz estimate. Later,
in \cite{CKSTT}, the authors developed a powerful interaction Morawetz estimate and studied the scattering theory for the cubic defocusing Schr\"odinger equation in a space with regularity less than energy space. The  interaction Morawetz estimate reads
\begin{equation}\label{int-Mor}
\big\||\nabla|^{-\frac{d-3}{4}}u\big\|_{L_{t,x}^4}^4\leq C \|u(0)\|^2_{L^2_x}\|u(t)\|^2_{L_t^\infty\dot{H}^{\frac12}},
\end{equation}
if $u$ solves a defocusing nonlinear Schr\"odinger equation without potentials. The interaction Morawetz estimate \eqref{int-Mor} is so powerful
that the authors of \cite{CKSTT07} proved the scattering theory for energy-critical defocusing Schr\"odinger and the authors of \cite{TVZ} presented a different proof of the scattering for energy subcritical cases which simplified the argument of \cite{GV}. \vspace{0.1cm}

In the proof of the scattering theory of nonlinear Schr\"odinger equation,  the Strichartz estimates and interaction Morawetz  estimates are two fundamental tools. For our purpose of studying the long time behavior of solution to \eqref{equ:nlsab4}, in spirit of the papers mentioned above, it is natural to establish the Strichartz and the interaction Morawetz estimates for the linear and nonlinear Schr\"odinger equations \eqref{equ:nlsab4} respectively.
However, the Aharonov-Bohm magnetic potential is scaling-critical and the perturbation is non-trivial, so
the situation becomes more complicated. In $\R^4$,  the magnetic field is $B=\text{curl}{\A}$ and the trapping component $B_\tau=\frac{x}{|x|}\wedge B$ is an obstruction to the dispersion which was observed by \cite{AFVV, FV}. Fortunately, the trapping component $B_\tau=0$ which is generated by the special potential ${\A}$ in \eqref{p:A}. Hence we can obtain the classical Morawetz estimates from the virial identities
derived by \cite{FV}. While for interaction Morawetz estimates, one needs a positivity of the quantity like $B(x)\frac{x-y}{|x-y|}$
which seems impossible for any $y$.
Colliander et.al in \cite{CCL} proved the interaction Morawetz estimates for the defocusing nonlinear Schr\"odinger with magnetic potentials. A key step is to use the method of \cite{FV} to bound the quantity involving $B(x)\frac{x-y}{|x-y|}$ when the decay of the magnetic potential ${\A}$
is fast enough. However, for ${\A}$ considered here, the decay is not enough to satisfy their assumption
\cite[(1.38)]{CCL,FV}, therefore the method breaks down.  The authors in \cite{ZZ18} obtained the interaction Morawetz estimate for Schr\"odinger with  Aharonov-Bohm potential in 2D,
but assumed that the solution is radial. In 2D, ${\A}\cdot\nabla f=0$  if $f$ is a radial function.
 Hence, at present, we have to study the scattering theory via the classical Morawetz estimate by following the method in Ginibre-Velo \cite{GV} .

 \vspace{0.2cm}

Now we state our main theorems.
\begin{definition}
Let $\dot{H}^1_{{\A}}(\R^4)$ be the completion of $C_c^\infty(\R^4\setminus\{0\})$ with respect to the norm
\begin{equation}\label{norm1}
\|f\|_{\dot{H}^1_{{\A}}(\R^4)}:=\Big(\int_{\R^4}|\nabla_{{\A}} f(x)|^2\mathrm{d}x\Big)^{\frac12},
\end{equation}
where $\nabla_{\A}:=-i\nabla+{\A}(x)$. Then the inhomogenous space is
\begin{equation}\label{norm2}
H^1_{{\A}}(\R^4)=\dot{H}^1_{{\A}}(\R^4)\cap L^2(\R^4).
\end{equation}
\end{definition}

\begin{theorem}
\label{thm:scatter}
Let $u_0(x)\in H^1_{{\A}}(\R^4)$. Then, there exists a unique global solution $u\in C(\R,H^1_{{\A}}(\R^4))$
to \eqref{equ:nlsab4}. Moreover, the solution
$u$ scatters in the sense that there exists $u_\pm\in H^1_{{\A}}(\R^4)$ such that
\begin{equation}\label{equ:sctt}
  \lim_{t\to\pm\infty}\big\|u(t,\cdot)-e^{it\mathcal{L}_{{\A}}}u_\pm\big\|_{H^1_{{\A}}(\R^4)}=0.
\end{equation}
\end{theorem}

As mentioned above, in order to establish a scattering theory, we have to prove Strichartz estimates for our model. Fanelli-Vega \cite{FV} obtained the Strichartz estimates  for the wave equation with a non-trapping electromagnetic potential with almost Coulomb decay by establishing a family of virial-type identities for $n\geq3$. However, the method is not applicable to magnetic Schr\"odinger equation because they only proved the weak dispersive estimate hold in $\dot{\mathcal{H}}^{\frac12}$ instead of $L^2$.
Later in \cite{AFVV}, they proved the Strichartz estimates for Schr\"odinger when the potentials are almost critical which do not
include the potential \eqref{p:A}. Their method can prove the endpoint Strichartz estimates with derivatives,
but can not cover the usual endpoint Strichartz estimates without derivatives, see \cite[(1.13), (1.14)]{AFVV}.

To achieve the purpose of scattering theory in Theorem \ref{thm:scatter}, we need the chain rules associated with $\LL_{\A}$ for
differential operators of non-integer order, hence we can derive the desired nonlinear estimates. It is well known that
 the traditional Littlewood-Paley theory gives the proof of Leibniz (=product) and chain rules for
differential operators of non-integer order.  For example, if $1<p<\infty$ and $s>0$, then
$$
\| f g \|_{H^{s,p}(\R^n)} \lesssim \| f \|_{H^{s,p_1}(\R^n)} \| g \|_{L^{p_2}(\R^n)} + \| f \|_{L^{p_3}(\R^n)}\| g \|_{H^{s,p_4}(\R^n)}
$$
whenever $\frac1p=\frac1{p_1}+\frac1{p_2}=\frac1{p_3}+\frac1{p_4}$.  For a textbook presentation of these theorems and original references, see \cite{Taylor:Tools}. Our strategy is to prove a  boundedness of generalized Riesz transform in $L^p$
\begin{equation}\label{riesz1}
  \big\|(-\Delta)^\frac{s}{2}\mathcal{L}_{\A}^{-\frac{s}{2}}f\big\|_{L^p(\R^4)}\leq C\|f\|_{L^p(\R^4)}
\end{equation}
and reversed Riesz transform
\begin{equation}\label{riesz1'}
  \big\|\mathcal{L}_{\A}^{\frac{s}{2}}(-\Delta)^{-\frac{s}{2}}f\big\|_{L^p(\R^4)}\leq C\|f\|_{L^p(\R^4)}.
\end{equation}
Hence we obtain
\begin{equation}\label{equ:esonorm1-1}
  \big\|(-\Delta)^\frac{s}{2}f\big\|_{L^p(\R^4)}\simeq \big\|\mathcal{L}_{\A}^\frac{s}{2}f\big\|_{L^p(\R^4)}.
\end{equation}
To prove these, we need a boundedness of heat kernel.
\begin{theorem}\label{thm:heat}
Let $\mathcal{L}_{{\A}}$ be in \eqref{LA}. Then there exists a constant $C$ such that
\begin{equation}\label{equ:heatlA}
  \big|e^{-t\mathcal{L}_{\A}}(x,y)\big|\leq C t^{-2}e^{-\frac{|x-y|^2}{4t}},\quad t>0.
\end{equation}
\end{theorem}
From the heat kernel estimates and our previous work \cite{KMVZZ} with Killip, Visan and Miao,  one can obtain the equivalence of Sobolev norms
\begin{equation}\label{equ:esonorm1}
  \big\|(-\Delta)^\frac{s}{2}f\big\|_{L^p(\R^4)}\simeq \big\|\mathcal{L}_{\bf A}^\frac{s}{2}f\big\|_{L^p(\R^4)},\quad 1< p<\frac{4}{s}
\end{equation}
therefore, there holds chain rules
\begin{equation}\label{chain:LA}
\| f g \|_{H^{s,p}_{\bf A}(\R^4)} \lesssim \| f \|_{H^{s,p_1}_{\bf A}(\R^4)} \| g \|_{L^{p_2}(\R^4)} + \| f \|_{L^{p_3}(\R^4)}\| g \|_{H^{s,p_4}_{\bf A}(\R^4)},
\end{equation}
for $s>0$, $1<p,p_j<\tfrac{4}{s}(j=1,2,3,4)$
and
$\frac1p=\frac1{p_1}+\frac1{p_2}=\frac1{p_3}+\frac1{p_4}$.\vspace{0.2cm}

Our paper is organized as follows. Section 2 presents the representation of the spectral measure associated with $\LL_{\A}$. Section 3 is devoted to the proofs of the Strichartz estimates. In Section 4, we establish a scattering theory for the defocusing nonlinear subcritical Schr\"odinger equation \eqref{equ:nlsab4}.

\section{The spectral measure}

In this section, we first reduce the operator $H_{A,2}$ to an operator $\LL_{\A}$ given in \eqref{LA} below by making change of
variables. By modifying the method of \cite{FZZ, GYZZ}, we prove the representations of Schr\"odinger and wave propagators, spectral measure associated with the operator $\LL_{\A}$.
\subsection{The operator} From \eqref{Fj} and \eqref{H-N} with $N=2$,
we see
\begin{equation}\label{H-2}
H_{A,2}=\sum_{j=1}^2 (-i\nabla_j+{\bf F}_j)^2,
\end{equation}
where $\vec x_1=(x_{11}, x_{12}), \vec x_2=(x_{21},x_{22})\in\R^2$ and
\begin{align*}
&\nabla_{\vec x_1}=(\partial_{x_{11}},\partial_{x_{12}}),\quad {\bf F}_1=\alpha\Big(-\frac{x_{12}-x_{22}}{|\vec x_1-\vec x_2|},\frac{x_{11}-x_{21}}{|\vec x_1-\vec x_2|}\Big), \\
&\nabla_{\vec x_2}=(\partial_{x_{21}},\partial_{x_{22}}), \quad {\bf F}_2=\alpha\Big(-\frac{x_{22}-x_{12}}{|\vec x_1-\vec x_2|},\frac{x_{21}-x_{11}}{|\vec x_1-\vec x_2|}\Big).
\end{align*}
Let
\begin{equation*}
{\vec y}_1=\frac{\vec x_1-\vec x_2}{\sqrt{2}},\quad {\vec y}_2=\frac{\vec x_1+\vec x_2}{\sqrt{2}},
\end{equation*}
then
\begin{equation*}
\nabla_{{\vec x}_1}=\frac1{\sqrt{2}}(\nabla_{{\vec y}_1}+\nabla_{{\vec y}_2}),\quad\nabla_{\vec x_2}=\frac1{\sqrt{2}}(\nabla_{{\vec y}_2}-\nabla_{{\vec y}_1}),
\end{equation*}
and
\begin{align}\label{equ:ha2la}
H_{A,2}=&\Big(-i\nabla_{\vec y_1}+\alpha\big(-\frac{y_{12}}{|\vec y_1|^2},\frac{y_{11}}{|\vec y_1|^2}\big)\Big)^2-\Delta_{\vec y_2}\\\nonumber
=&\Big(-i\big(\nabla_{\vec y_1},\nabla_{\vec y_2}\big)+\alpha\big(-\frac{y_{12}}{|\vec y_1|^2},\frac{y_{11}}{|\vec y_1|^2},0,0\big)\Big)^2.
\end{align}
In fact
\begin{align*}
H_{A,2}=& (-i\nabla_{\vec{x}_1}+{\bf F}_1)^2+(-i\nabla_{\vec{x}_2}+{\bf F}_2)^2\\
=&\tfrac12\big(-i\nabla_{\vec{y}_1}-i\nabla_{\vec{y}_2}+\alpha\big(-\tfrac{y_{12}}{|\vec{y}_1|^2},\tfrac{y_{11}}{|\vec{y}_1|^2}\big)\big)^2+
\tfrac12\big(i\nabla_{\vec{y}_1}-i\nabla_{\vec{y}_2}-\alpha\big(-\tfrac{y_{12}}{|\vec{y}_1|^2},\tfrac{y_{11}}{|\vec{y}_1|^2}\big)\big)^2\\
=&\tfrac12\big(-i\nabla_{\vec{y}_1}+\alpha\big(-\tfrac{y_{12}}{|\vec{y}_1|^2},\tfrac{y_{11}}{|\vec{y}_1|^2}\big)-i\nabla_{\vec{y}_2}\big)^2+
\tfrac12\big(-i\nabla_{\vec{y}_1}+\alpha\big(-\tfrac{y_{12}}{|\vec{y}_1|^2},\tfrac{y_{11}}{|\vec{y}_1|^2}\big)+i\nabla_{\vec{y}_2}\big)^2\\
=&\tfrac12\big(-i\nabla_{\vec{y}_1}+\alpha\big(-\tfrac{y_{12}}{|\vec{y}_1|^2},\tfrac{y_{11}}{|\vec{y}_1|^2}\big)\big)^2
-\big(-i\nabla_{\vec{y}_1}+\alpha\big(-\tfrac{y_{12}}{|\vec{y}_1|^2},\tfrac{y_{11}}{|\vec{y}_1|^2}\big)\big)\cdot i\nabla_{\vec{y}_2}
-\tfrac12\Delta_{\vec{y}_2}\\
&+\tfrac12\big(-i\nabla_{\vec{y}_1}+\alpha\big(-\tfrac{y_{12}}{|\vec{y}_1|^2},\tfrac{y_{11}}{|\vec{y}_1|^2}\big)\big)^2
+\big(-i\nabla_{\vec{y}_1}+\alpha\big(-\tfrac{y_{12}}{|\vec{y}_1|^2},\tfrac{y_{11}}{|\vec{y}_1|^2}\big)\big)\cdot i\nabla_{\vec{y}_2}
-\tfrac12\Delta_{\vec{y}_2}\\
=&\Big(-i\nabla_{\vec y_1}+\alpha\big(-\frac{y_{12}}{|\vec y_1|^2},\frac{y_{11}}{|\vec y_1|^2}\big)\Big)^2-\Delta_{\vec y_2}\\
=&\Big(-i\nabla_{\vec y_1}+\alpha\big(-\frac{y_{12}}{|\vec y_1|^2},\frac{y_{11}}{|\vec y_1|^2}\big),-i\nabla_{\vec y_2}\Big)^2\\
=&\Big(-i\big(\nabla_{\vec y_1},\nabla_{\vec y_2}\big)+\alpha\big(-\frac{y_{12}}{|\vec y_1|^2},\frac{y_{11}}{|\vec y_1|^2},0,0\big)\Big)^2,
\end{align*}
which implies \eqref{equ:ha2la}. Hence, if we denote $y=(\vec{y}_1,\vec{y}_2)\in\R^4$ and
$$v(\vec{y}_1,\vec{y}_2)=u\big(\tfrac{\vec{y}_1+\vec{y}_2}{\sqrt{2}},-\tfrac{\vec{y}_1-\vec{y}_2}{\sqrt{2}}\big)\quad\text{or}\quad
u(\vec{x}_1,\vec{x}_2)=v(\vec{y}_1,\vec{y}_2),$$
then,
\begin{align}\label{changes}
  H_{A,2} u(\vec x_1, \vec x_2)=&
  \LL_{\A} v(\vec y_1, \vec y_2),
\end{align}
 with
\begin{equation*}
\LL_{{\A}}=\Big(-i\nabla_y+{\A}(y)\Big)^2, \quad {\A}(y)=\alpha\Big(\frac {-y_{12}}{|y_{11}|^2+|y_{12}|^2},\frac {y_{11}}{|y_{11}|^2+|y_{12}|^2},0,0\Big)
\end{equation*}
 which is analogue of \eqref{AB}. Both of them are effected in the first two variables by the magnetic potential.
 Without confusing, from now on, we briefly write $\LL_{\A}$ to consist with \eqref{AB} as follows:
 \begin{equation}\label{LA}
\LL_{{\A}}=\Big(-i\nabla_x+{\A}(x)\Big)^2, \quad x=(x_{1}, x_{2}, x_{3}, x_{4})\in\R^4\setminus\{0,0, x_3, x_4\},
\end{equation}
with
\begin{equation}\label{p:A}
{\A}(x)=\alpha\left(-\frac{x_2}{x_1^2+x_2^2},\frac{x_1}{x_1^2+x_2^2}, 0, 0\right),\quad \alpha\in\R.
\end{equation}
This operator can be regraded as a one-particle model but with more singular
Aharonov-Bohm type potential in $\R^4$.\vspace{0.2cm}

We remark that the operator \eqref{LA} with magnetic potential \eqref{p:A} was studied in \cite{FKLV} from the Hardy inequality viewpoint and in \cite{FV} from the virial identities, see the special model in \cite[(1.19)]{FV} and
\cite[(3)]{FKLV}.
However,  the Strichartz estimates for dispersive equations with potential \eqref{p:A} have not been proved yet even though the Strichartz estimates obtained in \cite{FV} for a bit faster decaying ${\A}$.
It is known that the magnetic field $B=\text{curl} {\A}=(0,0,0,\delta)$ with  $\delta$ denoting the Dirac delta function.
Heuristic, since the dispersion obstruction $B_\tau=\frac{x}{|x|}\wedge B$ (the trapping component of magnetic field arising from ${\A}$) vanishes,
the Strichartz estimate should hold. Therefore we can modify the method of \cite{FZZ, GYZZ} to study the operator \eqref{LA}.

\subsection{The Schr\"odinger  kernel}
In this subsection, we modify the argument of \cite{GYZZ}, in which we construct the spectral measure for a similar model in $\R^2$,  to construct the kernel of Schr\"odinger  operator $e^{it\mathcal{L}_{{\A}}}$.
\begin{proposition}\label{prop:ker-S} Let $x=(r\cos\theta,r\sin\theta,x_3, x_4)$, $y=(\bar{r}\cos\bar{\theta},\bar{r}\sin\bar{\theta},y_3, y_4)$, and let $K(t;r,\theta,x_3, x_4;\bar{r},\bar{\theta},y_3, y_4)$ be the kernel of $e^{it\mathcal{L}_{{\A}}}$ with $\LL_{{\A}}$ being as in \eqref{LA}, then there holds
\begin{align}\label{kernel:S}
K(t;r,\theta,x_3,x_4;\bar{r},\bar{\theta},y_3,y_4)=&\frac1{4\pi}\frac1{(it)^{2}}e^{-\frac{|x-y|^2}{4it}}A_\alpha(\theta,\bar{\theta})
+\frac1{4\pi}\frac1{(it)^{2}}\\\nonumber
&\times e^{-\frac{r^2+\bar{r}^2+|x_3-y_3|^2+|x_4-y_4|^2}{4it}}\int^\infty_0
e^{-\frac{r\bar{r}\cosh s}{2it}}B_\alpha(s,\theta,\bar{\theta})\mathrm{d}s,\\\nonumber
:=&G(t;r,\theta,x_3,x_4;\bar{r},\bar{\theta},y_3,y_4)+D(t;r,\theta,x_3;\bar{r},\bar{\theta},y_3,y_4),
\end{align}
where $A_\alpha(\theta,\bar{\theta})$ and $B_\alpha(s,\theta,\bar{\theta})$ are respectively
\begin{align}\label{A-al}
A_\alpha(\theta,\bar{\theta})&=e^{i\alpha(\theta-\bar{\theta})}\big[1_{[0,\pi]}(|\theta-\bar{\theta}|)
+e^{-2\pi i\alpha}1_{[\pi,2\pi]}(|\theta-\bar{\theta}|)\big],
\end{align}
and
\begin{align}\label{B-al}
B_\alpha(s,\theta,\bar{\theta})=&-\Big[\sin(|\alpha|\pi)e^{-|\alpha|s}+\sin(\alpha\pi)\nonumber\\
&\times\frac{(e^{-s}-\cos(\theta-\bar{\theta}+\pi))\sinh\alpha s-i\sin(\theta-\bar{\theta}+\pi)\cosh\alpha s}{\cosh s-\cos(\theta-\bar{\theta}+\pi)}\Big].
\end{align}
\end{proposition}
\begin{proof}
We prove \eqref{kernel:S} by following the idea in \cite{GYZZ}. First, we obtain the fundamental solution of Schr\"odinger  operator $e^{it\mathcal{L}_{{\A}}}$.
 In the cylindrical coordinates
\begin{equation}\label{CY}
x_1=r\cos\theta,\quad x_2=r\sin\theta,\quad x_3=x_3,\quad x_4=x_4,
\end{equation}
from \eqref{p:A}, we write
\begin{equation}\label{LA'}
\begin{split}
\mathcal{L}_{{\A}}&=-\Delta +2i{\A}(x)\cdot\nabla +|{\A}(x)|^2\\
&=-\partial_r^2-\frac1r\partial_r+\frac{(i\partial_\theta+\alpha)^2}
{r^2}-\partial_{x_3}^2-\partial_{x_4}^2.
\end{split}
\end{equation}
For each $k\in \Z$, $Y_k(\theta):=(2\pi)^{-1/2}e^{-ik\theta}$ satisfies
\begin{align}\label{eig-Y}
\begin{cases}
(i\partial_\theta+\alpha)^2Y_k(\theta)=(k+\alpha)^2Y_k(\theta),\quad on  \ \ \mathbb{S}^1,\\
\int_{\mathbb{S}^1}|Y_k(\theta)|^2\mathrm{d}\theta=1,
\end{cases}
\end{align}
hence, the eigenfunctions $\{Y_k(\theta)\}_{k\in\Z}$ construct a complete orthonormal basis in $L^2(\mathbb{S}^1)$. Moreover, let $x'=(x_3,x_4)\in\R^2$, we can write $f(x)\in L^2(\mathbb{R}^4)$ in the form of
\begin{equation}\label{exp-f}
f(x)=f(r,\theta,x')=\sum_{k\in\Z} a_k(r,x')Y_k(\theta),
\end{equation}
where
$$a_k(r,x')=\int^{2\pi}_0 f(r,\theta,x')Y_k(\theta).$$
Therefore, from \eqref{LA'}, \eqref{eig-Y} and \eqref{exp-f}, we know
\begin{equation}\label{2.4}
\begin{split}
\mathcal{L}_{{\A}}f
&=\sum_{k\in\Z}\Big(A_{k,\alpha}a_k(r,x')-\partial_{x_3}^2a_k(r,x')-\partial_{x_4}^2a_k(r,x')\Big)Y_k(\theta)
\end{split}
\end{equation}
where
\begin{equation}\label{Ak-al}
\begin{split}
A_{k,\alpha}=-\partial_r^2-\frac1r\partial_r+\frac{(k+\alpha)^2}
{r^2}.
\end{split}
\end{equation}
Let $\nu=\nu(k)=|k+\alpha|$, and recall  Hankel transform  (e.g. see \cite{BPSS, Taylor}) of $\nu$ order defined by
\begin{equation}\label{HT}
(H_\nu f)(\rho,\theta,x')=\int^\infty_0J_\nu(r\rho)f(r,\theta,x')r
\mathrm{d}r,\quad\forall f\in L^2(\mathbb{R}^4),
\end{equation}
where the Bessel function
\begin{equation}\label{bess}
J_\nu(r)=\frac{(r/2)^\nu}{\Gamma(\nu+\frac{1}{2})\Gamma(1/2)}\int_{-1}^1e^{isr}(1-s^2)^{(2\nu-1)/2}\mathrm{d}s,
\quad \nu>-\frac12, r>0.
\end{equation}
For our purpose, we recall some properties of the Hankel transform, see \cite{BPSS, Taylor}.
\begin{lemma}[Hankel transform]\label{lem:hankel}
Let $\mathcal{H}_\nu$ be the Hankel transform in \eqref{HT} and $A_\nu:=-\partial^2_r-\frac{1}{r}\partial_r+\frac{\nu^2}{r^2}$. Then
\begin{item}
\item $\mathrm{(1)}$ $\mathcal{H}_\nu=\mathcal{H}^{-1}_\nu,$

\item $\mathrm{(2)}$ $\mathcal{H}_\nu$ is self-adjoint, i.e.$\quad \mathcal{H}_\nu=\mathcal{H}^{*}_\nu$,

\item $\mathrm{(3)}$ $\mathcal{H}_\nu$ is an $L^2$ isometry, i.e. $\|\mathcal{H}_\nu f\|_{L^2(\R^4)}=\|f\|_{L^2(\R^4)},$

\item $\mathrm{(4)}$ $\mathcal{H}_\nu(A_\nu f)(\rho,\theta)=\rho^2(\mathcal{H}_\nu f)(\rho,\theta),$ for $f\in L^2.$

\end{item}

\end{lemma}

In the cylindrical coordinates \eqref{CY}, if the initial data
$$u_0(x)=f(r,\theta,x')=\sum_{k\in\Z} a_k(r,x')Y_k(\theta),$$
then $u(t,x)=e^{it\LL_{{\A}}}f$ satisfies
$$u(t,x)=u(t;r,\theta,x')=\sum_{k\in\Z} u_k(t;r,x')Y_k(\theta),$$
where $u_k(t;r,x')$ solves
\begin{equation}\label{Sch}
\begin{cases}
i\partial_t u_k+A_{k,\alpha} u_k-\partial_{x_{3}}^2u_k-\partial_{x_4}^2u_k=0,\quad (t,r,x')\in\R\times (0,\infty)\times \R^2,\\
u_k(0,r, x')= a_k(r,x') .
\end{cases}
\end{equation}

By taking Fourier transform in variables $x'$ and Hankel transform in variable $r$ and using Lemma \ref{lem:hankel}, and let $\xi=(\xi_3,\xi_4)\in\R^2$, we obtain
\begin{equation}\label{ord}
\begin{cases}
i\partial_t\tilde{u}_k(t;\rho,\xi)+\rho^2\tilde{u}_k(t;\rho,\xi)+
|\xi|^2\tilde{u}_k(t;\rho,\xi)=0,\\
\tilde{u}_k(0;\rho,\xi)=\tilde{a}_k(\rho,\xi).
\end{cases}
\end{equation}
where
$$\tilde{u}_k(t;\rho,\xi)=[H_{\nu(k)}\hat{u}_k](t;\rho,\xi),\qquad\tilde{a}_k(\rho,\xi)
=[H_{\nu(k)}\hat{a}_k](\rho,\xi)$$
and $\hat{a}(\xi)=\int_{\R^2} e^{-ix'\cdot\xi} a(x') dx'$.
The solution for ordinary differential equation \eqref{ord} is
$$\tilde{u}_k(t;\rho,\xi)=\tilde{a}_k(\rho,\xi)e^{it(\rho^2+|\xi|^2)}.$$
By using Lemma \ref{lem:hankel} again, we obtain
$$u(t;r,\theta,x')=\int_{\R^2}\int_0^\infty\int_0^{2\pi}K(t;r,\theta,x';\bar{r},\bar{\theta},y')f(\bar{r},\bar{\theta},y')\,\bar{r}
\mathrm{d}\bar{r}\mathrm{d}\bar{\theta} dy',$$
where
\begin{equation}
\begin{split}
&K(t;r,\theta,x';\bar{r},\bar{\theta},y')=\sum_{k\in\Z}Y_k(\theta)\overline{Y_k(\bar{\theta})}
\\ &\times\int_{\R^2} \int_0^\infty e^{i(x'-y')\cdot\xi} e^{it(\rho^2+|\xi|^2)}
J_{\nu(k)}(r\rho)J_{\nu(k)}(\bar{r}\rho)\rho\mathrm{d}\rho \mathrm{d}\xi .
\end{split}
\end{equation}
 Noting that
$$\int_{\R^2} e^{-it|\xi|^2}e^{-ix\cdot\xi}\mathrm{d}\xi=\frac \pi {it}e^{-\frac {|x|^2}{4it}}$$
and $Y_k(\theta)=(2\pi)^{-1/2}e^{-ik\theta}$, we rewrite $K(t;r,\theta,x';\bar{r},\bar{\theta},y')$ as
\begin{align}\label{3.7'}
&K(t;r,\theta,x';\bar{r},\bar{\theta},y')\nonumber\\
=&\frac 1 {2it}e^{-\frac {|x'-y'|^2}{4it}}
\sum_{k\in \mathbb{Z}}e^{-ik(\theta-\bar{\theta})}\int_0^\infty J_{\nu(k)}(r\rho)J_{\nu(k)}(\bar{r}\rho)e^{it\rho^2}\rho\mathrm{d}\rho,
\nonumber\\
=&\frac1{2it}e^{-\frac {|x'-y'|^2}{4it}}\sum_{k\in \mathbb{Z}}e^{-ik(\theta-\bar{\theta})}\times\lim_{\varepsilon\searrow0}\frac{e^{-\frac{r^2+\bar{r}^2}{4(\varepsilon+it)}}}{2(\varepsilon+it)}
I_{\nu(k)}\left(\frac{r\bar{r}}{2(\varepsilon+it)}\right),
\end{align}
where the Weber identity \cite{Taylor} is used in the last equality and the term $I_{\nu(k)}(z)$ (\cite{Watson}) is given by
$$I_{\nu(k)}(z)=\frac 1\pi\int^\pi _0e^{z\cos s}\cos \nu s\mathrm{d}s-\frac {\sin \nu\pi}\pi\int^\infty _0e^{-z\cosh s}e^{-s\nu}\mathrm{d}s.$$
Recall $\nu=|k+\alpha|$, similarly as in \cite{FZZ,GYZZ}, we get
\begin{align}\label{3.8'}
&\frac 1\pi\sum_{k\in \mathbb{Z}}e^{-ik(\theta-\bar{\theta})}\int^\pi _0e^{z\cos s}\cos \nu s\mathrm{d}s\nonumber\\
=&\frac 1\pi\times
\begin{cases}
e^{z\cos(\theta-\bar{\theta})}e^{i\alpha(\theta-\bar{\theta})},\quad|\theta-\bar{\theta}|<\pi,\\
e^{z\cos(\theta-\bar{\theta})}e^{i\alpha(\theta-\bar{\theta}-2\pi)},\quad\pi<|\theta-\bar{\theta}|<2\pi,
\end{cases}
\end{align}
and
\begin{align}\label{3.9'}
&\frac 1\pi\sum_{k\in \mathbb{Z}}e^{-ik(\theta-\bar{\theta})}\sin\nu\pi\int^\infty _0e^{-z\cosh s}e^{-s\nu}\mathrm{d}s\nonumber\\
=&\frac 1\pi\int_0^\infty e^{-z\cosh s}\Big[\sin(|\alpha|\pi)e^{-|\alpha|s}+\sin(\alpha\pi)\nonumber\\&\times
\frac{(e^{-s}-\cos(\theta-\bar{\theta}+\pi))\sinh\alpha s-i\sin(\theta-\bar{\theta}+\pi)\cosh\alpha s}{\cosh s-\cos(\theta-\bar{\theta}+\pi)}\Big]\mathrm{d}s.
\end{align}
By noticing that
$$(x_1-y_1)^2+(x_2-y_2)^2=r^2+\bar{r}^2-2r\bar{r}\cos(\theta-\bar{\theta})$$
and collecting \eqref{3.7'}-\eqref{3.9'} together, we finally prove Proposition \ref{prop:ker-S}.
\end{proof}
Decay estimates follow from the representation of fundamental solution \eqref{kernel:S}.
\begin{proposition}\label{prop:disper}
There exists a constant $C$ such that
\begin{equation}\label{DS}
|e^{it\mathcal{L}_{{\A}}}(x,y)|\leq C |t|^{-2},\quad \forall t\in\R\setminus\{0\}.
\end{equation}
\end{proposition}

\begin{proof}
Indeed,  from \cite{FZZ}, we know
\begin{equation}\label{in-B}
|A_\alpha(\theta,\bar{\theta})|+\int^\infty_0|B_\alpha(s,\theta,\bar{\theta})|\leq C,
\end{equation}
then combining with \eqref{kernel:S}, dispersive estimates \eqref{DS} hold.
\end{proof}

\subsection{The representation of spectral measure} To prove the dispersive estimates for wave equation, we need the representation of the spectral measure associated with $\LL_{\A}$.
\begin{proposition}\label{prop:spect} Let $x=(r\cos\theta,r\sin\theta,x')$, $y=(\bar{r}\cos\bar{\theta},\bar{r}\sin\bar{\theta},y')$ where $x'=(x_3,x_4)\in\R^2$ and $y'=(y_3,y_4)\in\R^2$. Let $d E_{\sqrt{\mathcal{L}_{\A}}}(\lambda;x,y)$ be the spectral measure kernel associated with Schr\"odinger operator $\mathcal{L}_{\A}$ given by \eqref{LA}. Then
\begin{align}\label{kernel:sp}
d E_{\sqrt{\mathcal{L}_{\A}}}(\lambda;x,y)=&\frac{\lambda^3}{\pi^2}\sum_{\pm}a_{\pm}(\lambda|x-y|)e^{\pm i\lambda|x-y|}\times A_\alpha(\theta,\bar{\theta})\nonumber\\
&+\frac{\lambda^3}{\pi^2}\int^\infty_0\sum_{\pm}a_{\pm}(\lambda|\vec{n}|)e^{\pm i\lambda|\vec{n}|}\times B_\alpha(s,\theta,\bar{\theta})\mathrm{d}s,
\end{align}
where $a_{\pm}(r)$ satisfies
\begin{equation}\label{bean:a}
|\partial_r^k a_{\pm}(r)|\leq C_k(1+r)^{-\frac32-k},\quad\forall k\geq0,
\end{equation}
 and $A_\alpha(\theta,\bar{\theta})$, $B_\alpha(s,\theta,\bar{\theta})$ are as in \eqref{A-al} and \eqref{B-al} respectively, and
  \begin{equation}\label{equ:n}
 \vec{n}=\vec{n}_s=(r+\bar{r},\sqrt{2r\bar{r}(\cosh s-1)},x_3-y_3, x_4-y_4).
 \end{equation}

\end{proposition}
\begin{proof}
Note that for each $z\in\mathbb{C}$ satisfying $\Im z>0$, there holds
\begin{equation}\label{4.1}
\frac1{s-z}=i\int^\infty_0e^{-ist}e^{izt}\mathrm{d}t,\quad \forall s\in\mathbb{R}.
\end{equation}
This together with \eqref{kernel:S} shows
\begin{align}\label{4.2}
(\mathcal{L}_{\A}-(\lambda^2+i0))^{-1}=&i\lim_{\varepsilon\searrow0}\int^\infty_0e^{-it\mathcal{L}_{\A}}
e^{it(\lambda^2+i\varepsilon)}\mathrm{d}t\nonumber\\
=&i\lim_{\varepsilon\searrow0}\int^\infty_0G(t;r,\theta,x_3,x_4;\bar{r},\bar{\theta},y_3,y_4)
e^{it(\lambda^2+i\varepsilon)}\mathrm{d}t\nonumber\\
&+i\lim_{\varepsilon\searrow0}\int^\infty_0D(t;r,\theta,x_3,x_4;\bar{r},\bar{\theta},y_3,y_4)
e^{it(\lambda^2+i\varepsilon)}\mathrm{d}t.
\end{align}
As did in \cite{GYZZ}, noting that
$$\int_{\R^4}e^{-ix\cdot\xi} e^{-it|\xi|^2}\;d\xi=\frac{\pi}{(it)^2}e^{-\frac{|x|^2}{4it}},$$
we get for $z=\lambda^2+i\epsilon$ with $\epsilon>0$
\begin{align*}
 & \int_0^\infty \frac{e^{-\frac{|x-y|^2}{4it}} }{(it)^2}e^{itz}\;dt
  =\frac{1}{\pi}\int_0^\infty\int_{\R^4}e^{-i(x-y)\cdot\xi} e^{-it|\xi|^2}\;d\xi \,e^{itz}\;dt\\
  =&\frac{1}{\pi}\int_{\R^4}e^{-i(x-y)\cdot\xi} \int_0^\infty e^{-it(|\xi|^2-z)}\;dt\;d\xi
  =\frac{1}{i\pi}\int_{\R^4}\frac{e^{-i(x-y)\cdot\xi}}{|\xi|^2-z}\;d\xi,
\end{align*}
and
\begin{equation}\label{equ:intr1r223}
  \int_0^\infty  \frac{e^{-\frac{r_1^2+r_2^2+|x'-y'|^2}{4it}} }{(it)^2} e^{-\frac{r_1r_2}{2it}\cosh s} e^{itz}\;dt=\frac{1}{i\pi}\int_{\R^4}\frac{e^{-i{\vec{n}}\cdot\xi}}{|\xi|^2-z}\;d\xi,
\end{equation}
where ${\vec {n}}=(r_1+r_2, \sqrt{2r_1r_2(\cosh s-1)}, x'-y')$.
Then from \eqref{kernel:S}, we have
\begin{equation}
i\lim_{\varepsilon\searrow0}\int^\infty_0G(t;r,\theta,x';\bar{r},\bar{\theta},y')
e^{it(\lambda^2+i\varepsilon)}\mathrm{d}t=\frac1{4\pi^2}\int_{\mathbb{R}^4}\frac{e^{-i(x-y)\cdot\xi}}
{|\xi|^2-(\lambda^2+i0)}\mathrm{d}\xi\times A_\alpha(\theta,\bar{\theta}),
\end{equation}
and
\begin{align}\label{4.4}
&i\lim_{\varepsilon\searrow0}\int^\infty_0 D(t;r,\theta,x';\bar{r},\bar{\theta},y')
e^{it(\lambda^2+i\varepsilon)}\mathrm{d}t\nonumber\\
=&\frac1{4\pi^2}\int^\infty_0\int_{\mathbb{R}^4}\frac{e^
{-i(r+\bar{r},\sqrt{2r\bar{r}(\cosh s-1)},x_3-y_3,x_4-y_4)\cdot\xi}}
{|\xi|^2-(\lambda^2+i0)}\mathrm{d}\xi\times B_\alpha(s,\theta,\bar{\theta})\mathrm{d}s.
\end{align}
Therefore, collecting \eqref{4.2}-\eqref{4.4} together and recalling \eqref{equ:n}, we get
\begin{align}\label{inc-res}
(\mathcal{L}_{\A}-(\lambda^2+i0))^{-1}=&\frac1{4\pi^2}\int_{\mathbb{R}^4}\frac{e^{-i(x-y)\cdot\xi}}
{|\xi|^2-(\lambda^2+i0)}\mathrm{d}\xi\times A_\alpha(\theta,\bar{\theta})\nonumber\\
&+\frac1{4\pi^2}\int^\infty_0\int_{\mathbb{R}^4}\frac{e^
{-i\vec{n}\cdot\xi}}
{|\xi|^2-(\lambda^2+i0)}\mathrm{d}\xi\times B_\alpha(s,\theta,\bar{\theta})\mathrm{d}s.
\end{align}
Observing that $\overline{A_{-\alpha}}=A_{\alpha}$ and $\overline{B_{-\alpha}}=B_{\alpha}$, then we obtain
\begin{align}\label{out-res}
(\mathcal{L}_{\A}-(\lambda^2-i0))^{-1}=&\frac1{4\pi^2}\int_{\mathbb{R}^4}\frac{e^{-i(x-y)\cdot\xi}}
{|\xi|^2-(\lambda^2+i0)}\mathrm{d}\xi\times A_\alpha(\theta,\bar{\theta})\nonumber\\
&+\frac1{4\pi^2}\int^\infty_0\int_{\mathbb{R}^4}\frac{e^
{-i\vec{n}\cdot\xi}}
{|\xi|^2-(\lambda^2-i0)}\mathrm{d}\xi\times B_\alpha(s,\theta,\bar{\theta})\mathrm{d}s.
\end{align}
By using the Stone's formula, we deduce that
\begin{equation}\label{4.7}
\mathrm{d}E_{\sqrt{\mathcal{L}_{\A}}}(\lambda;x,y)=\frac\lambda{i\pi}
\big[(\mathcal{L}_{\A}-(\lambda^2+i0))^{-1}-(\mathcal{L}_{\A}-(\lambda^2-i0))^{-1}\big]\mathrm{d}\lambda.
\end{equation}
Combining \eqref{inc-res}-\eqref{4.7}, we further write
\begin{align*}
&\mathrm{d}E_{\sqrt{\mathcal{L}_{\A}}}(\lambda;x,y)\\
=&\frac1i\frac\lambda{4\pi^3}\int_{\mathbb{R}^4}
e^{-i(x-y)\cdot\xi}(\frac1{|\xi|^2-(\lambda^2-i0)}-\frac1{|\xi|^2-(\lambda^2+i0)})\mathrm{d}\xi
\times A_\alpha(\theta,\bar{\theta})\\
&+\frac1i\frac\lambda{4\pi^3}\int^\infty_0\left[\int_{\mathbb{R}^4}
e^{-i\vec{n}\cdot\xi}(\frac1{|\xi|^2-(\lambda^2-i0)}-\frac1{|\xi|^2-(\lambda^2+i0)})\mathrm{d}\xi\right]
\times B_\alpha(s,\theta,\bar{\theta})\mathrm{d}s.
\end{align*}
Similarly to \cite{GYZZ} again, we observe
the fact that
 the  Poisson kernel is an approximation to the identity which implies that, for any reasonable function $m(x)$
\begin{equation}
\begin{split}
m(x)&=\lim_{\epsilon\to 0^+}\frac1\pi \int_{\R} {\rm Im}\Big(\frac{1}{x-(y+i\epsilon)}\Big) m(y)dy
\\&=\lim_{\epsilon\to 0^+}\frac1\pi \int_{\R} \frac{\epsilon}{(x-y)^2+\epsilon^2} m(y)dy,
\end{split}
\end{equation}
then we have
\begin{equation}\label{id-spect}
\begin{split}
&\lim_{\epsilon\to 0^+}\frac{\lambda}{i\pi }\int_{\R^4} e^{-ix\cdot\xi}\Big(\frac{1}{|\xi|^2-(\lambda^2+i\epsilon)}-\frac{1}{|\xi|^2-(\lambda^2-i\epsilon)}\Big) d\xi\\
=&\lim_{\epsilon\to 0^+} \frac{\lambda}{\pi }\int_{\R^4} e^{-ix\cdot\xi}\Im\Big(\frac{1}{|\xi|^2-(\lambda^2+i\epsilon)}\Big)d\xi\\
=&\lim_{\epsilon\to 0^+} \frac{\lambda}{\pi }\int_{0}^\infty \frac{\epsilon}{(\rho^2-\lambda^2)^2+\epsilon^2} \int_{|\omega|=1} e^{-i\rho x\cdot\omega} d\sigma_\omega  \, \rho^{3} d\rho\\
=& \lambda^3 \int_{|\omega|=1} e^{-i\lambda x\cdot\omega} d\sigma_\omega.
\end{split}
\end{equation}
On the other hand, from \cite[Theorem 1.2.1]{sogge}, we also note that
\begin{equation}
\begin{split}
\int_{\mathbb{S}^{3}} e^{-i x\cdot\omega} d\sigma(\omega)=\sum_{\pm}  a_\pm(|x|) e^{\pm i|x|},
\end{split}
\end{equation}
where
\begin{equation}
\begin{split}
| \partial_r^k a_\pm(r)|\leq C_k(1+r)^{-\frac32-k},\quad k\geq 0.
\end{split}
\end{equation}
Therefore we finally obtain representation \eqref{kernel:sp}.

\end{proof}

\section{Strichartz estimates}
In this section, we prove Theorem \ref{thm:stri-S} and Theorem \ref{thm:stri-w} about the Strichartz estimates.

\subsection{The proof of Theorem \ref{thm:stri-S}}
From the dispersive estimates in Proposition \ref{prop:disper} and the $L^2$-estimate (obtained from the mass conservation law for Schr\"odinger equation or the unitary property of $e^{it\LL_{\A}}$),  the abstract method of Keel-Tao \cite{KT}
directly shows
\begin{theorem}\label{thm:stri0}
Let $\mathcal{L}_{{\A}}$ be in \eqref{LA} and let
$u(t,x)$ be a solution of Schr\"odinger equation
\begin{equation}\label{equ:S'}
\begin{cases}
i\partial_t u+\mathcal{L}_{{\A}} u=0,\qquad (t,x)\in\R\times\R^4\setminus\{0,0, x_3, x_4\}\\
u(0,x)=u_0(x),\quad x\in\R^4.
\end{cases}
\end{equation}
 Then there exists a constant $C$ such that
\begin{equation}\label{stri-LA}
\|u(t,x)\|_{L^{q}_{t}(\R;L_x^{r}(\mathbb{R}^4))}\leq C\|u_0\|_{L^2(\mathbb{R}^4)},
\end{equation}
where
 $(q,r)\in\Lambda_0^S$ defined in \eqref{adm}.

\end{theorem}

\begin{remark} The Strichartz estimates \eqref{stri-LA} can be extended to the solution of
\begin{equation}\label{homo}
\begin{cases}
i\partial_t u+\mathcal{L}_{{\A},a} u=0,\qquad (t,x)\in\R\times\R^4\setminus\{0\},\\
u(0,x)=u_0(x),\quad x\in\R^4,
\end{cases}
\end{equation}
where the electromagnetic operator
$\LL_{{\A},a}=\Big(-i\nabla+{\A}(x)\Big)^2+a(\hat{x})|x|^{-2}$ and $x\in\R^4\setminus\{0\},\, \hat{x}\in \mathbb{S}^{3}$ and
$\min_{\hat{x}\in \mathbb{S}^3} a(\hat{x})>-1$. Indeed, this is a consequence of  \eqref{stri-LA} and the local smoothing estimates
\begin{equation}
\||x|^{-1}u(t,x)\|_{L^{2}_{t}(\R;L_x^{2}(\mathbb{R}^4))}\leq C\|u_0\|_{L^2(\mathbb{R}^4)},
\end{equation}
which is implied by
the weighted resolvent estimates in Barcel\'o-Vega-Zubeldia\cite{BVZ}
\begin{equation}
\sup_{\sigma\in\C\setminus[0,\infty)}\big\||x|^{-1}(\LL_{{\A},a}-\sigma)^{-1}|x|^{-1}\big\|_{L^{2}_{t}(\R;L_x^{2}(\mathbb{R}^4))}\leq C.
\end{equation}

\end{remark}

\begin{remark}  The method also works for dimension three in which
\begin{equation*}
{\A}(x)=\alpha\left(-\frac{x_2}{|x_1|^2+|x_2|^2},\frac{x_1}{|x_1|^2+|x_2|^2}, 0\right),\quad \alpha\in\R.
\end{equation*}
In this case, the endpoint Strichartz estimate \eqref{stri-LA} with $(q,r)=(2,6)$ is a generalization of \cite[(1.14)]{AFVV}.
\end{remark}

Now we prove Theorem \ref{thm:stri-S} by using \eqref{stri-LA}. Let $u(t,\vec x_1, \vec x_2)$ solve \eqref{equ:S}, from \eqref{changes}, then
\begin{equation}
v(t;\vec y_1, \vec y_2)= u\Big(t; \frac{\vec y_1+\vec y_2}{\sqrt{2}},\frac{\vec y_2-\vec y_1}{\sqrt{2}}\Big)
\end{equation}
solves \eqref{equ:S'}.
Therefore, by \eqref{equ:HAB} and \eqref{stri-LA}, we obtain for $(q,r)\in\Lambda_0^S$ defined in \eqref{adm}
\begin{equation}\label{u-v}
\begin{split}
\|u(t,\vec x_1,\vec x_2)\|_{L^q_t(\R;L^r_x(\R^4))}&\leq C\|v(t;\vec y_1, \vec y_2)\|_{L^q_t(\R;L^r_x(\R^4))}
\\&\leq C\|v_0(\vec y_1, \vec y_2)\|_{L^2(\R^4)}
\leq C\|u_0(\vec x_1, \vec x_2)\|_{L^2(\R^4)}.
\end{split}
\end{equation}
Thus, we complete the proof of Theorem \ref{thm:stri-S}.

\subsection{The proof of Theorem \ref{thm:stri-w}} By the same argument of \eqref{u-v}, to prove Theorem \ref{thm:stri-w}, it suffices to establish the Strichartz estimates
for the wave equation
\begin{equation}\label{equ:w'}
\begin{cases}
\partial_{tt} u+\mathcal{L}_{{\A}} u=0,\qquad (t,x)\in\R\times\R^4\setminus\{0,0, x_3, x_4\}\\
u(0,x)=f(x),\qquad \partial_t u(0,x)= g(x)\qquad x\in\R^4.
\end{cases}
\end{equation}

To this goal, we first prove the localized dispersive estimates for wave \eqref{equ:w'}.
\begin{proposition}[Dispersive estimate for wave]\label{prop:dis-w} Let $\LL_{\A}$ be given in \eqref{LA} and
let $\phi\in C_c^\infty([1/2, 2])$ take value in $[0,1]$ such that $1=\sum\limits_{k\in\Z}\phi(2^{-k}\lambda)$. Assume $f=\phi(2^{-k}\sqrt{\LL_{\A}})f$ with $k\in \Z$, then there exists a constant $C$ independent of $t$ and $k\in \Z$ such that
\begin{equation}\label{dispersive}
\begin{split}
\|e^{it\sqrt{\LL_{\A}}} f\|_{L^\infty(\R^4)}\leq C 2^{\frac52 k}(2^{-k}+|t|)^{-\frac32}\|f\|_{L^1(\R^4)}.
\end{split}
\end{equation}
\end{proposition}

\begin{proof} With Proposition \ref{prop:spect} in hand, we use stationary phase argument to prove Proposition \ref{prop:dis-w}.
We write
 \begin{equation*}
e^{it\sqrt{\LL_{\A}}}f=\int_{\R^4} \int_0^\infty e^{it\lambda} \phi(2^{-k}\lambda) dE_{\sqrt{\LL_{\A}}}(\lambda; x, y) f(y) dy.
  \end{equation*}
Now we aim  to show kernel estimate
  \begin{equation}\label{est:k}
\Big|  \int_0^\infty e^{it\lambda} \phi(2^{-k}\lambda) dE_{\sqrt{\LL_{\A}}}(\lambda; x, y) \Big|\leq C 2^{\frac52 k}(2^{-k}+|t|)^{-\frac32}.
\end{equation}
For this purpose, from Proposition \ref{prop:spect}, it suffices to show
  \begin{equation}\label{est:k1}
\Big|  \int_0^\infty e^{it\lambda} \phi(2^{-k}\lambda)\lambda^3 a_\pm(\lambda |x-y|)e^{\pm i\lambda |x-y|} d\lambda
A_{\alpha}(\theta,\bar{\theta}) \Big|\leq C 2^{\frac52 k}(2^{-k}+|t|)^{-\frac32},
\end{equation}
and
  \begin{equation}\label{est:k2}
  \begin{split}
\Big|  \int_0^\infty e^{it\lambda} \phi(2^{-k}\lambda)\lambda^3 \int_0^\infty a_\pm(\lambda |\vec{n}_s|)e^{\pm i\lambda |\vec{n}_s|}
&B_{\alpha}(s,\theta,\bar{\theta}) ds \,d\lambda\Big|\\&\leq C 2^{\frac52 k}(2^{-k}+|t|)^{-\frac32},
\end{split}
\end{equation}
where $a_\pm$ satisfies \eqref{bean:a},
 $\vec{n}_s$ is in \eqref{equ:n}, and $A_\alpha(\theta,\bar{\theta})$, $B_\alpha(s,\theta,\bar{\theta})$ are in \eqref{A-al} and \eqref{B-al}.
Since $a_\pm$ satisfies \eqref{bean:a}, let $r=|x-y|$ or $|\vec{n}_s|$, hence
\begin{equation}\label{bean'}
 \big|\partial_\lambda^N [a_\pm(\lambda r)]\big|\leq C_N \lambda^{-N}(1+\lambda r)^{-\frac{3}2},\quad N\geq 0.
\end{equation}
We first prove \eqref{est:k1}. If $2^k|t\pm |x-y||\leq 1$, by \eqref{bean'}, we directly estimate
\begin{equation*}
\begin{split}
&\Big|  \int_0^\infty e^{it\lambda} \phi(2^{-k}\lambda)\lambda^3 a_\pm(\lambda |x-y|)e^{\pm i\lambda |x-y|} d\lambda\Big|\\ \leq&
C\int_{2^{k-1}}^{2^{k+1}}\lambda^3(1+\lambda
|x-y|)^{-3/2}d\lambda\\\leq&
C2^{4k}(1+2^k|x-y|)^{-3/2}.
\end{split}
\end{equation*}
 If $2^k|t\pm |x-y||\geq 1$, we use \eqref{bean'} again and perform $N$-times integration by parts to obtain
\begin{equation*}
\begin{split}
&\Big|  \int_0^\infty e^{it\lambda} \phi(2^{-k}\lambda)\lambda^3 a_\pm(\lambda |x-y|)e^{\pm i\lambda |x-y|} d\lambda\Big|\\
\leq& \Big|\int_0^\infty \left(\frac1{
i(t\pm |x-y|)}\frac\partial{\partial\lambda}\right)^{N}\big(e^{i(t\pm |x-y|)\lambda}\big)
\phi(2^{-k}\lambda)\lambda^3 a_\pm(\lambda |x-y|) d\lambda\Big|\\
 \leq&
C_N|t\pm |x-y||^{-N}\int_{2^{k-1}}^{2^{k+1}}\lambda^{3-N}(1+\lambda
|x-y|)^{-3/2}d\lambda\\
\leq&
C_N2^{k(4-N)}|t\pm |x-y||^{-N}(1+2^k|x-y|)^{-3/2}.
\end{split}
\end{equation*}
It follows that
\begin{equation}\label{dispersive2}
\begin{split}
&\Big|  \int_0^\infty e^{it\lambda} \phi(2^{-k}\lambda)\lambda^3 a_\pm(\lambda |x-y|)e^{\pm i\lambda |x-y|} d\lambda\Big|\\
\leq&
C_N2^{4k}\big(1+2^k|t\pm |x-y||\big)^{-N}(1+2^k |x-y|)^{-3/2}.
\end{split}
\end{equation}
If $|t|\sim |x-y|$, we see \eqref{est:k1}.
Otherwise, we have $|t\pm |x-y||\geq c|t|$ for some small constant
$c$, choose $N=1$ and $N=0$,  and then use geometric mean argument to obtain \eqref{est:k1}.\vspace{0.2cm}

We next prove \eqref{est:k2}. We follow the same lines to obtain
  \begin{equation}\label{est:k2-1}
  \begin{split}
\Big|  \int_0^\infty e^{it\lambda} \phi(2^{-k}\lambda)\lambda^3 &\int_0^\infty a_\pm(\lambda |\vec{n}_s|)e^{\pm i\lambda |\vec{n}_s|}
B_{\alpha}(s,\theta,\bar{\theta}) ds \,d\lambda\Big|\\&\leq C 2^{\frac52 k}(2^{-k}+|t|)^{-\frac32} \int_0^\infty |B_{\alpha}(s,\theta,\bar{\theta})| ds.
\end{split}
\end{equation}
From \eqref{in-B}, we have
\begin{equation}\label{bound-A}
\int_0^\infty |B_{\alpha}(s,\theta,\bar{\theta})| ds\leq C,
\end{equation}
which implies \eqref{est:k2}. Therefore we prove \eqref{est:k}, hence \eqref{dispersive} holds.

\end{proof}

We also need the following proposition to achieve our goal.
\begin{proposition}[A kernel estimate]\label{prop:L2} Let $\LL_{\A}$ be given in \eqref{LA} and
let $\psi\in C_c^\infty([1/2, 2])$ and take value in
$[0,1]$. For $k\in \Z$ and any $K\geq 0$, there exists a constant $C$ independent of $k\in \Z$ such that
\begin{equation}\label{est:L2-ker}
\begin{split}
\Big|\int^\infty_0 \psi(2^{-k}\lambda) dE_{\sqrt{\LL_{\A}}}(\lambda;x,y)\Big|\lesssim \frac{2^{4k}}{(1+2^k|x-y|)^{K}}.
\end{split}
\end{equation}
\end{proposition}

\begin{proof}

On one hand, by  using scaling, we see
\begin{equation}\label{scaling-L2}
\begin{split}
\Big|\int^\infty_0 \psi(2^{-k}\lambda) dE_{\sqrt{\LL_{\A}}}(\lambda;x,y)\Big|=2^{4k}\Big|\int^\infty_0 \psi(\lambda) dE_{\sqrt{\LL_{\A}}}(\lambda;2^kx,2^ky)\Big|.
\end{split}
\end{equation}
On the other hand,  by  using Proposition \ref{prop:spect} and integration by parts K times, we have
\begin{equation*}
\begin{split}
\Big|\int^\infty_0 \psi(\lambda) dE_{\sqrt{\LL_{\A}}}(\lambda;x,y)\Big|\lesssim  \frac{1}{(1+|x-y|)^{K}}+\int_0^\infty \frac{1}{(1+|\vec{n}_s|)^{K}} |B_{\alpha}(s,\theta,\bar{\theta})| ds.
\end{split}
\end{equation*}
From \eqref{equ:n}, we observe that $|\vec{n}_s|^2\geq (r_1+r_2)^2+|x_3-y_3|^2+|x_4-y_4|^2\geq |x-y|^2$.
Thus it follows from \eqref{bound-A} that
\begin{equation*}
\begin{split}
\Big|\int^\infty_0 \psi(\lambda) dE_{\sqrt{\LL_{\A}}}(\lambda;x,y)\Big|\lesssim  \frac{1}{(1+|x-y|)^{K}}.
\end{split}
\end{equation*}
This together with \eqref{scaling-L2} gives \eqref{est:L2-ker}.

\end{proof}

Next we need the Littlewood-Paley square function inequality associated with the operator $\LL_{\A}$.
This can be done by following the proof of \cite[Proposition 2.2]{FZZ}, once we show the heat kernel estimate in Theorem \ref{thm:heat}.
Even $\LL_{\A}$ is slightly different from the operator $\LL_{{\A},0}$ in \cite{FZZ},
we can use the argument to prove \eqref{equ:heatlA}. Indeed, if one replaces the variable $t$ by $t/i$, by the analytic continuous argument, we can obtain \eqref{equ:heatlA} from Proposition \ref{prop:ker-S}. We omit the detail of the proof of Theorem \ref{thm:heat}.

Therefore, by the standard argument deriving Littlewood-Paley theory from the Gaussian boundedness of heat kernel (e.g. see \cite{FZZ}), we have
\begin{proposition}[LP square function inequality]\label{prop:squarefun} Let $\{\phi_k\}_{k\in\mathbb Z}$ be in Proposition \ref{prop:dis-w} and let $\LL_{\A}$ be given in \eqref{LA}.
Then for $1<p<\infty$,
there exist constants $c_p$ and $C_p$ depending on $p$ such that
\begin{equation}\label{square}
c_p\|f\|_{L^p(\R^4)}\leq
\Big\|\Big(\sum_{k\in\Z}|\phi_k(\sqrt{\LL_{{\A}}})f|^2\Big)^{\frac12}\Big\|_{L^p(\R^4)}\leq
C_p\|f\|_{L^p(\R^4)}.
\end{equation}

\end{proposition}

For each $k\in \mathbb{Z}$, let $\phi_k(\lambda)=\phi(2^{-k}\lambda)$, we define
\begin{align}\label{propa}
U_k(t)=\int^\infty_0e^{it\lambda} \phi_k(\lambda) dE_{\sqrt{\LL_{\A}}}(\lambda;x,y),
\end{align}
where $dE_{\sqrt{\LL_{\A}}}(\lambda;x,y)$ is the spectral measure in Proposition \ref{prop:spect}.
Hence by spectral property (e.g. \cite[Lemma 5.3]{HZ}), we see
$$U_k(t)U_k^\ast(s)=\int^\infty_0e^{i(t-s)\lambda} \big(\phi_k(\lambda)\big)^2 dE_{\sqrt{\LL_{\A}}}(\lambda;x,y).$$
On one hand, by Proposition \ref{prop:dis-w}, we have
\begin{equation}
\|U_k(t)U_k^\ast(s)\|_{L^1\to L^\infty}\lesssim 2^{5k}(2^{-k}+|t-s|)^{-3}.
\end{equation}
On the other hand, by duality argument, we have
\begin{equation}
\|U(t)\|^2_{L^2\to L^2}=\|U_k(t)U_k^\ast(t)\|_{L^2\to L^2}\lesssim 1.
\end{equation}
Indeed, this is implied by
$$\Big|\int^\infty_0\big(\phi_k(\lambda)\big)^2 dE_{\sqrt{\LL_{\A}}}(\lambda;x,y)\Big|\lesssim \frac{2^{4k}}{(1+2^k|x-y|)^{N}}\in L^1(\R^4),$$
which follows from Proposition \eqref{prop:L2}.

Then using Keel-Tao's argument (also see \cite{KT}), for $(q,r)\in \Lambda^w_s$, we get
\begin{equation}
\|U_k(t)f\|_{L_{t}^{q}(\mathbb{R}L^r(\mathbb{R}^4))}\lesssim 2^{ks}\|f\|_{L^2(\mathbb{R}^4)}.
\end{equation}
 Let $f_j(x)=\phi(2^{-j}\sqrt{\LL_{\A}})f$, then
\begin{equation}\label{d-c}
e^{it\sqrt{\mathcal{L}_{\A}}}f(x)=\sum_{k\in\Z} U_k(t)f(x)=\sum_{j\in\Z}\sum_{k\in\Z} U_k(t)f_j(x).
\end{equation}
By using \eqref{square} and Minkowski's inequality, we show that
\begin{align}\label{4.21}
\|e^{it\sqrt{\mathcal{L}_{\A}}}f\|_{L_{t}^{q}(\mathbb{R}L^r(\mathbb{R}^4))}
\lesssim&\Big\|\big(\sum_{j\in\mathbb{Z}}|\sum_{k\in\mathbb{Z}}U_k(t)f_j|^2\big)^{1/2}\Big\|_{L_{t}^{q}(\mathbb{R}L^r(\mathbb{R}^4))}
\nonumber\\
\lesssim&\Big(\sum_{j\in\mathbb{Z}}\big\|\sum_{k\in\mathbb{Z}}U_k(t)f_j\big\|_{L_{t}^{q}(\mathbb{R}L^r(\mathbb{R}^4))}
^2\Big)^{1/2}.
\end{align}\label{4.22}
We observe that $\phi_k(\sqrt{\mathcal{L}_{\A}})f_j$ vanishes when $|j-k|\geq 10$, thus
\begin{align}\label{4.23}
\Big(\sum_{j\in\mathbb{Z}}\big\|\sum_{k\in\mathbb{Z}}U_k(t)f_j\big\|_{L_{t}^{q}(\mathbb{R}L^r(\mathbb{R}^4))}
^2\Big)^{1/2}\lesssim&\Big(\sum_{j\in\mathbb{Z}}\sum_{|j-k|\leq10}\|U_k(t)f_j\|_{L_{t}^{q}(\mathbb{R}L^r(\mathbb{R}^4))}
^2\Big)^{1/2}\nonumber\\
\lesssim&\Big(\sum_{j\in\mathbb{Z}}2^{js}\|f_j\|_{L^{2}(\mathbb{R}^4)}
^2\Big)^{1/2}=\|f\|_{\dot{H}_{\A}^s(\mathbb{R}^4)}.
\end{align}
Since
\begin{equation}\label{4.19}
u(t,x)=\frac{e^{it\sqrt{\mathcal{L}_{\A}}}+e^{-it\sqrt{\mathcal{L}_{\A}}}}2
f(x)+\frac{e^{it\sqrt{\mathcal{L}_{\A}}}-e^{-it\sqrt{\mathcal{L}_{\A}}}}{2i\sqrt{
\mathcal{L}_{\A}}}g(x),
\end{equation}
 we finally prove the Strichartz estimates
\begin{equation}
\|u(t,x)\|_{L_{t}^{q}(\mathbb{R}L^r(\mathbb{R}^4))}\lesssim\|f\|_{\dot{H}_{\A}^s(\mathbb{R}^4)}+\|g\|_{\dot{H}_{\A}^{s-1}(\mathbb{R}^4)}.
\end{equation}
By the equivalence of Sobolev norms \eqref{equ:esonorm1}, we finally obtain \eqref{stri-w}. The proof of Theorem \ref{thm:stri-w} is now complete.

\section{Scattering theory}
 We divide this section into three parts. The first part proves the global well-posedness for the solution to the equation \eqref{equ:nlsab4}. Then we show the classical Morawetz estimate by using the virial identity in \cite[Theorem1.2]{FV}. Finally, we prove Theorem \ref{thm:scatter} by following the idea in \cite{GV}.

\subsection{Global well-posedness}
By interpolating \eqref{DS} and mass conservation, we obtain the following dispersive estimates.
\begin{proposition}[Dispersive Estimate]\label{Dispersive}
Let $\alpha\in \R$ and $2\leq p\leq+\infty$. Then we have the following estimate
\begin{equation}\label{DS'}
\|e^{it\mathcal{L}_{{\A}}}f\|_{L^p(\R^4)}\leq C |t|^{-4(\frac12-\frac1p)}\|f\|_{L^{p'}(\R^4)},\quad \forall t\in\R\setminus\{0\}.
\end{equation}
for some constant $C=C(\alpha,p)>0$ which does not depend on $t$ and $f$.
\end{proposition}

Then Strichartz estimates follow from a direct application of Keel-Tao's argument \cite{KT}.

\begin{proposition}[Strichartz estimate including inhomogeneous term]\label{Stri:inhomo}
Let $\mathcal{L}_{{\A}}$ be in \eqref{LA},
$u(t,x)$ be a solution of nonlinear Schr\"odinger equation
\begin{equation}\label{inhomo}
\begin{cases}
i\partial_t u+\mathcal{L}_{\A} u=F(t,x),\qquad (t,x)\in I\times\R^4,\\
u(t_0,x)=u_0,\quad x\in\R^4,
\end{cases}
\end{equation}
then for admissible pairs $(q,r),(\bar{q},\bar{r})\in\Lambda_0^S$,
there exists a constant $C$ such that
\begin{equation}\label{str1}
\|u(t,x)\|_{L^{q}_{t}(I\times L_x^r(\mathbb{R}^4))}\leq C\|u_0\|_{L^2(\mathbb{R}^4)}+\big\|
    F\big\|_{L_t^{\bar q'}L_x^{\bar r'}(I\times\R^4)}.
\end{equation}
\end{proposition}
\begin{remark}
From \eqref{str1} and \eqref{equ:esonorm1},  we obtain
\begin{equation}\label{instr}
\big\|\nabla_{\A}
u\big\|_{L_t^qL_x^r(I\times\R^4)}\lesssim\|u_0\|_{\dot
H_{\A}^1(\R^4)}+
    \big\|\nabla_{\A}
    F\big\|_{L_t^{\bar q'}L_x^{\bar r'}(I\times\R^4)},
\end{equation}
and
\begin{equation}\label{instr1}
\big\|\nabla
u\big\|_{L_t^qL_x^r(I\times\R^4)}\lesssim\|u_0\|_{\dot
H^1(\R^4)}+
    \big\|\nabla
    F\big\|_{L_t^{\bar q'}L_x^{\bar r'}(I\times\R^4)},\quad r<4.
\end{equation}
\end{remark}

With dispersive estimates \eqref{Dispersive}, Strichartz estimates \eqref{str1}-\eqref{instr1} and chain rules of non-integer order for operator $\LL_{\bf A}$  \eqref{chain:LA} in hand, we derive the following local well-posedness theory.
\begin{proposition}[Local well-posedness theory]\label{localwell}
Let $u_0\in H^1_{\A}(\R^4)$. Then there exists
$T=T(\|u_0\|_{H^1_{\A}})>0$ such that the equation
\eqref{equ:nlsab4} with initial data $u_0$ has a unique solution $u$ with
\begin{equation}\label{small}
u,~\nabla_{\bf A} u\in C(I; L^2(\R^4))\cap L_t^{q}(I;L^r(\R^4)),\quad
I=[0,T),
\end{equation}
with $(q,r)\in\Lambda_0^S $.
\end{proposition}

\begin{proof}
We argue by the standard Banach fixed point argument. To this end, let
$$X(I)=C(I; L^2(\R^4))\cap L_t^{q_0}(I;L^{r_0}(\R^4)),\quad (q_0,r_0)=\big(\tfrac{p+1}{p-1},p+1\big),$$
we consider the map $$\Phi:u \mapsto  e^{it\mathcal{L}_{\A}} u_0  - i \int_0^t e^{i(t-s)\mathcal{L}_{\A}}
 |u(s)|^{p-1}
 u(s)\mathrm{d}s$$
 on the complete metric space
\begin{align*}
B:=\Bigl\{ u,~\nabla_A u\in C(I; L^2(\R^4))\cap L_t^{q_0}(I;L^{r_0}(\R^4)): \,
\|u\|_{X(I)},\|\nabla_{\A}u\|_{X(I)}
\leq 2C \|u_0\|_{H^1_{\A}} \Bigr\}
\end{align*}
and the metric
$$
d(u,v):= \| u -v \|_{L_t^{q_0}(I; L^{r_0}(\R^4))}.
$$
The constant $C$ depends only on the dimension and $\alpha$, and it
reflects implicit constants in the Strichartz and Sobolev embedding
inequalities.  We need to prove that the operator $\Phi$ is
well-defined on $B$ and is a contraction map under the metric $d$
for $I$.

Throughout the proof, all spacetime norms will be on $I\times\R^4$.
To see that $\Phi$ maps the ball $B$ to itself, we use the
Strichartz inequality in Theorem \ref{Stri:inhomo} and Sobolev embedding theorem and equivalent norms \eqref{equ:esonorm1} to arrive at
\begin{align*}
\|\Phi(u)\|_{X(I)}\leq&C\|u_0\|_{L_x^2}+C\big\||u|^{p-1}u\big\|_{L_t^{q'_0}L_x^{r'_0}}\\
\leq&C\|u_0\|_{L_x^2}+CT^{1-\frac2{q_0}}\|u\|_{L^\infty(I,L^{r_0})}^{p-1}\|u\|_{L^{q_0}(I,L^{r_0})}
\\ \leq&C\|u_0\|_{L_x^2}+CT^{1-\frac2{q_0}}\||\nabla|u\|_{L^\infty(I,L^2)}^{p-1}\|u\|_{L^{q_0}(I,L^{r_0})}\\ \leq&C\|u_0\|_{L_x^2}+CT^{1-\frac2{q_0}}\|\nabla_{\bf A}u\|_{L^\infty(I,L^2)}^{p-1}\|u\|_{L^{q_0}(I,L^{r_0})}\\
\leq&C\|u_0\|_{L_x^2}+CT^{1-\frac2{q_0}}\|\nabla_{\bf A}u\|_{X(I)}^{p-1}\|u\|_{X(I)}.
\end{align*}
 Noting that $\|u\|_{X(I)},~\|\nabla_{\A}u\|_{X(I)}\leq
 2C\|u_0\|_{H^1_{\A}}$, we see that for $u\in B$,
 $$\|\Phi(u)\|_{X(I)}\leq
 C\|u_0\|_{L_x^{2}}+CT^{1-\frac2{q_0}}\big(2C\|u_0\|_{H^1_{\A}}\big)^p.$$
 Since $1<p<3$, then
 \begin{equation}\label{range:p}
 1-\frac2{q_0}=1-\frac{2(p-1)}{p+1}>0.
 \end{equation}
Therefore, we take $T$ sufficiently small such that
$$T^{1-\frac2{q_0}}\big(2C\|u_0\|_{H^1_{\A}}\big)^p\leq \|u_0\|_{H^1_{\A}},$$
we have $\|\Phi(u)\|_{X(I)}\leq 2C\|u_0\|_{H^1_{\A}}$ for
$u\in B$. Similarly, by Strichartz estimate \eqref{instr}, chain rules \eqref{chain:LA}, equivalent norms \eqref{equ:esonorm1} and inequality \eqref{range:p}, we get
\begin{align*}
\big\|\nabla_{\A}
\Phi(u)\big\|_{X(I)}\leq&C\|\nabla_{\A}u_0\|_{L_x^2}+C\big\|\nabla_{\A}(|u|^{p-1}u)\big\|_{L^{q'_0}(I,L^{r'_0})}\\
\leq&C\|\nabla_{\A}u_0\|_{L_x^2}+CT^{1-\frac2{q_0}}\|u\|_{L^\infty(I,L^{r_0})}^{p-1}\|\nabla_{\bf A}u\|_{L^{q_0}(I,L^{r_0})}\\
\leq&C\|u_0\|_{\dot{H_{\bf A}^1}}+CT^{1-\frac2{q_0}}\||\nabla|u\|_{L^\infty(I,L^{2})}^{p-1}\|\nabla_{\bf A}u\|_{L^{q_0}(I,L^{r_0})}\\
\leq&C\|u_0\|_{L_x^2}+CT^{1-\frac2{q_0}}\|\nabla_{\bf A}u\|_{X(I)}^{p}.
\end{align*}
 This shows that $\Phi$ maps
the ball $B$ to itself.

 Finally, to prove that $\Phi$ is a
contraction, for $u,v\in B$, we argue as above
\begin{align*}
d\big(\Phi(u),\Phi(v)\big)\leq&C\big\||u|^{p-1}u-|v|^{p-1}v\big\|_{L_{t}^{q'_0}L_x^{r'_0}}\\
\leq&CT^{1-\frac2{q_0}}\|u-v\|_{L_t^{q_0} L_x^{r_0}}\big(\|\nabla_{\bf A}u\|_{L_t^\infty
L_x^{2}}^{p-1}+\|\nabla_{\bf A}v\|_{L_t^\infty L_x^{2}}^{p-1}\big)\\
\leq&2CT^{1-\frac2{q_0}}\big(2C\|u_0\|_{H^1_{\A}}\big)^{p-1}d(u,v)\\
\leq&\frac12d(u,v)
\end{align*}
by taking $T$ sufficiently small such that
$$2CT^{\frac12}\big(2C\|u_0\|_{H^1_{\A}}\big)^{p-1}\leq\frac12.$$

The standard fixed point argument gives a unique solution $u$ of
\eqref{equ:nlsab4} on $I\times\R^4$. By Strichartz estimate, we also get
the bound
$$\|u\|_{L_t^qL_x^r(I\times\R^4)}+\|\nabla_{\A}u\|_{L_t^qL_x^r(I\times\R^4)}\leq2C\|u_0\|_{H^1_{\A}(\R^4)},
\quad \forall~(q,r)\in\Lambda_0^S,\;r<4$$
\end{proof}

On the other hand, conservation laws for equation \eqref{equ:nlsab4} hold.

\begin{lemma}
[Conservation laws] If the solution $u$ for equation \eqref{equ:nlsab4} has sufficient decay at infinity and smoothness, it conserves the mass
\begin{equation}\label{Mass}
M(u)=\int_{\R^4}|u(t,x)|^2\mathrm{d}x=M(u_0)
\end{equation}
and energy
\begin{equation}\label{Energy}
E(u(t))=\frac12\int_{\R^4}|\nabla_{\A}u(t)|^2\mathrm{d}x+\frac1{p+1}\int_{\R^4}|u(t)|^{p+1}\mathrm{d}x=E(u_0).
\end{equation}
\end{lemma}

\begin{proof}
From equation \eqref{equ:nlsab4} and using the integration by parts, we see
\begin{align*}
\frac{\mathrm{d}}{\mathrm{d}t}&M(u(t))=2\Re\int_{\R^4}u_t\bar{u}\mathrm{d}x=2\Re\int_{\R^4}i(\LL_{\A}u-|u|^{p-1}u)\bar{u}\mathrm{d}x\\
&=-2\Im\int_{\R^4}(\LL_{\A}u-|u|^{p-1}u)\bar{u}\mathrm{d}x=2\Im\int_{\R^4}|\nabla_{\A}u|^2\mathrm{d}x+2\Im\int_{\R^4}|u|^{p+1}\mathrm{d}x=0,
\end{align*}
which implies \eqref{Mass}.

A similar computation yields
\begin{equation*}
\frac{\mathrm{d}}{\mathrm{d}t}E(u(t))=\Re\int_{\R^4}\nabla_{\A}u_t\overline{\nabla_{\A}u}+|u|^{p-1}\bar{u}u_t\mathrm{d}x=0,
\end{equation*}
which implies \eqref{Energy}.
\end{proof}

Using the mass and energy conservation laws, we obtain
\begin{equation}\label{mecl}
\|u(t)\|_{H_{\A}^1(\R^4)}^2\leq M(u_0)+2E(u_0).
\end{equation}
Therefore, the global well-posedness follows from the local well-posedness theory and mass and energy conservation.



\subsection{Morawetz estimate}\label{mor} In this subsection, we derive the classical Morawetz estimate from the virial identity in Theorem \ref{thm:viride}.

\begin{proposition}[Morawetz estimate]\label{prop:morest}
Assume that $u:\;I\times\R^4\to\C$ solves  the equation \eqref{equ:nlsab4}, then there holds
\begin{equation}\label{equ:morest}
  \int_I\int_{\R^4}\frac{|u(t,x)|^{p+1}}{|x|}\;dx\;dt\leq C\big(M(u_0), E(u_0)\big).
\end{equation}
\end{proposition}

Proposition \ref{prop:morest}  follows from Theorem \ref{thm:viride} with $a(x)=|x|$ below.
\begin{theorem}[Virial identity]\label{thm:viride}
Let $a:\mathbb{R}^4\rightarrow\R$ be a radial, real-valued multiplier, $a(x)=a(|x|)$and let
$$\Phi_a(t)=\int_{\mathbb{R}^4}a(x)|u|^2\mathrm{d}x.$$
Then for any solution $u$ of magnetic Schr\"odinger equation \eqref{equ:nlsab4} and initial datum $u_0\in L^2$, $\nabla_{\bf A}u_0\in L^2$, the following virial-type identity holds:
\begin{equation}\label{vir1}
\frac{\mathrm{d}^2}{\mathrm{d}t^2}\Phi_a(t)=4\int_{\mathbb{R}^4}\nabla_{\bf A}uD^2a\overline{\nabla_{\bf A}u}\mathrm{d}x-\int_{\mathbb{R}^4}|u|^2\Delta^2a\mathrm{d}x+\frac{2(p-1)}{p+1}\int_{\R^4}|u|^{p+1}\Delta  a\mathrm{d}x,
\end{equation}
where $$(D^2a)_{j,k}=\frac{\partial^2}{\partial x_j\partial x_k}a,\quad\Delta^2a=\Delta(\Delta a),$$
for $j,k=1,\cdots,4$ are respectively the Hessian matrix and the bi-Laplaican of $a$.
\end{theorem}

 \begin{proof}
 From equation \eqref{equ:nlsab4}, we see that
 \begin{equation*}
 u_t=i(\LL_{\bf A}+|u|^{p-1})u.
 \end{equation*}
 Then a simple calculation yields
 \begin{equation*}
\frac{\;d \Phi_a(t)}{\;dt}=i\langle u,[\LL_{\bf A},a]u\rangle,
\end{equation*}
and
\begin{equation}\label{der2}
\frac{\;d^2 \Phi_a(t)}{\;dt^2}=-\langle u,[\LL_{\bf A},[\LL_{\bf A},a]]u\rangle-(\langle u,[\LL_{\bf A},a]|u|^{p-1}u\rangle+\langle|u|^{p-1}u,[\LL_{\bf A},a]u\rangle).
 \end{equation}
 Consider the first term in \eqref{der2}, using the virial identity in \cite[Theorem1.2]{FV}, we arrive at
 \begin{equation}\label{vir-1}
-\langle u,[\LL_{\bf A},[\LL_{\bf A},a]]u\rangle=4\int_{\mathbb{R}^4}\nabla_{\bf A}uD^2a\overline{\nabla_{\bf A}u}\mathrm{d}x-\int_{\mathbb{R}^4}|u|^2\Delta^2a\mathrm{d}x.
 \end{equation}
 Recalling that the commutator
 $$[\LL_{\bf A},a]=2\nabla a\cdot\nabla_{\bf A}+\Delta a$$
  (see\cite[(2.7)]{FV}), we obtain for the reminder term by using integration by parts
 \begin{align}\label{vir-2}
 &-\langle u,[\LL_{\bf A},a]|u|^{p-1}u\rangle+\langle|u|^{p-1}u,[\LL_{\bf A},a]u\rangle\nonumber\\
 =&-2\langle u,\nabla  a\cdot\nabla_{\bf A}(|u|^{p-1}u)\rangle+2\langle |u|^{p-1}u,\nabla  a\cdot\nabla_{\bf A}u\rangle\nonumber\\
 =&\frac{2(p-1)}{p+1}\int_{\R^4}|u|^{p+1}\Delta  a\mathrm{d}x.
 \end{align}
 Collecting \eqref{der2}-\eqref{vir-2} together, we finally conclude \eqref{vir1}.
 \end{proof}

\subsection{Scattering Theory} In this subsection, we aim to establish a scattering theory by following the idea in \cite{GV}.
\subsubsection{Decay of potential energy}\label{sze}
We will utilize the classical Morawetz estimate in Proposition \ref{prop:morest} and dispersive estimates \eqref{Dispersive} and Strichartz estimates \eqref{str1}-\eqref{instr1} that we obtained in the last two subsections to show the decay of potential energy.
\begin{proposition}\label{prop:pedecay}
Let $u:\;\R\times\R^4\to \C$ be the global solution to \eqref{equ:nlsab4}. Then, there holds
\begin{equation}\label{equ:poedecay}
  \lim_{t\to\pm\infty}\|u(t,\cdot)\|_{L_x^{p+1}(\R^4)}=0.
\end{equation}
By interpolating with mass and energy, we obtain
\begin{equation}\label{equ:poedecayg}
 \lim_{t\to\pm\infty}\|u(t,\cdot)\|_{L_x^r(\R^4)}=0,\quad \forall\;2<r<p+1.
\end{equation}

\end{proposition}

\begin{proof}
We adapt Ginibre-Velo's method in \cite{GV} to show this proposition. For completeness, we give the detailed proof. We only need to show the positive time direction.

  We have
  \begin{equation}\label{equ:l42p}
    \int_{\R^4}|u(t,x)|^{p+1}\;dx= \int_{|x|\leq t\log t}|u(t,x)|^{p+1}\;dx+ \int_{|x|> t\log t}|u(t,x)|^{p+1}\;dx.
  \end{equation}

{\bf Step 1:} In this step, we will show
\begin{equation}\label{equ:step1large}
   \lim_{t\to\infty}\int_{|x|> t\log t}|u(t,x)|^{p+1}\;dx=0.
\end{equation}
  To do this, for given $M>0$, we define the smoothing function
  \begin{equation}\label{equ:defthem}
    \theta_M(x)= \begin{cases}
    \frac{|x|}{M}\quad \text{if}\quad |x|\leq M\\
    1\quad \text{if}\quad |x|\leq 2M.
    \end{cases}
  \end{equation}
Then, it is easy to check that
$$\theta_M\in W^{1,\infty}(\R^4),\quad \|\nabla \theta_M\|_\infty\lesssim \frac1M,\quad \theta_M u\in C(\R,H^1_{\bf A}(\R^4)).$$
By \eqref{equ:nlsab4} and a simple computation, we get
\begin{align*}
  \frac12\frac{d}{dt}\int_{\R^4}\theta_M(x)|u(t,x)|^2\;dx =&{\rm Re}\int_{\R^3}\theta_Mu_t\bar{u}\;dx\\
  =&{\rm Re}\int_{\R^4}\theta_M\bar{u}\big[i\LL_{\A}u+i|u|^{p-1}u\big]\;dx\\
  =&-{\rm Im}\int_{\R^4}\theta_M\bar{u}\LL_{\A}u\;dx\\
  =&{\rm Im}\int_{\R^4}(-i\nabla+{\bf A})\theta_M\cdot(i\nabla+{\bf A})u \bar{u}\;dx.
\end{align*}
  Hence
$$\Big| \frac{d}{dt}\int_{\R^4}\theta_M(x)|u(t,x)|^2\;dx \Big|\lesssim \big\|(-i\nabla+{\bf A})\theta_M  \big\|_\infty
\cdot\|u\|_{L_x^2}\cdot\big\|(i\nabla+{\bf A})u  \big\|_{L_x^2}\leq \frac{C}{M}.$$
  Integrating in time, we obtain
\begin{equation}\label{equ:thems}
  \int_{\R^4}\theta_M(x)|u(t,x)|^2\;dx\leq \frac{C}{M}t+\int_{\R^4}\theta_M(x)|u_0(x)|^2\;dx,\quad \forall\;t>0.
\end{equation}
  Taking $M=t\log t$, we have
\begin{align*}
 \int_{|x|> t\log t}|u(t,x)|^2\;dx \lesssim & \frac{1}{\log t}+\int_{\R^3}\theta_M(x)|u_0(x)|^2\;dx\\
 \lesssim & \frac{1}{\log t}+\int_{|x|\leq t\log t}\frac{|x|}{t\log t}|u_0(x)|^2\;dx+\int_{|x|> t\log t}|u_0(x)|^2\;dx\\
 \lesssim & \frac{1}{\log t}+\frac{1}{\sqrt{t\log t}}\int_{|x|\leq \sqrt{t\log t}}|u_0(x)|^2\;dx+\int_{|x|> \sqrt{t\log t}}|u_0(x)|^2\;dx\\
 \to& 0\qquad \text{as}\quad t\to\infty.
\end{align*}
Since $3<p+1<4$, so \eqref{equ:step1large} follows by  interpolating the above inequality with $\|u(t,x)\|_{L_x^4}\lesssim\|u(t,x)\|_{H_{\A}^1}\lesssim1$.

%

{\bf Step 2:} By Morawetz estimate \eqref{equ:morest}, we derive that for any $\epsilon>0,\;t>1$ and $\tau>0$, there exists $t_0>\max\{t,2\tau\}$ such that
\begin{equation}\label{equ:l4intt}
  \int_{t_0-2\tau}^{t_0}\int_{|x|\leq s\log s}|u(s,x)|^{p+1}\;dx\;ds<\epsilon.
\end{equation}
Actually, by Morawetz estimate \eqref{equ:morest}, we see that
\begin{align*}
\infty>&\int_1^\infty\int_{\R^4}\frac{|u(s,x)|^{p+1}}{|x|}\;dx\;ds>\int_1^\infty\frac1{s\log s}\int_{|x|\leq s\log s}|u(s,x)|^{p+1}\;dx\;ds\\
=&\sum_{k=0}^\infty\int_{t+2k\tau}^{t+2(k+1)\tau}\frac1{s\log s}\int_{|x|\leq s\log s}|u(s,x)|^{p+1}\;dx\;ds\\
\geq&\sum_{k=0}^\infty\frac1{(t+2(k+1)\tau)\log (t+2(k+1)\tau)}\int_{t+2k\tau}^{t+2(k+1)\tau}\int_{|x|\leq s\log s}|u(s,x)|^{p+1}\;dx\;ds.
\end{align*}
We note that
$$\sum_{k=0}^\infty\frac1{(t+2(k+1)\tau)\log (t+2(k+1)\tau)}=+\infty,$$
then there exists $k>0$ such that
$$\int_{t+2k\tau}^{t+2(k+1)\tau}\int_{|x|\leq s\log s}|u(s,x)|^{p+1}\;dx\;ds\leq\epsilon,$$
therefore, \eqref{equ:l4intt} holds by taking $t_0=t+2(k+1)\tau$.

{\bf Step 3:} We claim that for any $\epsilon>0$ and $b>0$, there exists $t_0>b$ such that
\begin{equation}\label{equ:t0clsoe}
  \sup_{s\in[t_0-b,t_0]}\|u(s,\cdot)\|_{L_x^{p+1}(\R^4)}<\epsilon.
\end{equation}
  Taking $t>\tau>0$, we have by Duhamel's formula
  \begin{align*}
    u(t)= & e^{it\LL_{\A}}u_0-i\int_0^{t-\tau}e^{i(t-s)\LL_{\A}}(|u|^{p-1}u)(s)\;ds-i\int_{t-\tau}^te^{i(t-s)\LL_{\A}}(|u|^{p-1}u)(s)\;ds\\
    \triangleq& v(t)+w(t,\tau)+z(t,\tau).
  \end{align*}
First, by density $\mathcal{S}(\R^4)\hookrightarrow L^p(\R^4)$ and the dispersive estimate, we obtain
\begin{equation}\label{equ:dispde}
  \lim_{t\to\infty}\|v(t)\|_{L^{p+1}_x(\R^4)}=0.
\end{equation}

{\bf Estimate of $w(t,\tau)$:} Using dispersive estimate and $\|u\|_{L_x^{p+1}}\lesssim\|u\|_{H^1}\lesssim1$, we get
\begin{equation}\label{equ:winfes}
  \|w(t,\tau)\|_{L_x^\infty}\lesssim\int_0^{t-\tau}(t-s)^{-2}\|u\|_{L_x^{p+1}}^{p+1}\;ds\lesssim \tau^{-1}.
\end{equation}
On the other hand, we have
$$w(t,\tau)=e^{i\tau\LL_{\A}}u(t-\tau)-e^{it\LL_{\A}}u_0.$$
This implies
$$\|w(t,\tau)\|_{L_x^2}\lesssim \|u_0\|_{L_x^2}\lesssim1.$$
Interpolating this with \eqref{equ:winfes} yields that
\begin{equation}\label{equ:wl4est}
  \|w(t,\tau)\|_{L_x^{p+1}}\leq K\tau^{-\frac{p-1}{p+1}}
\end{equation}
  for some $K>0$.

{\bf Estimate of $z(t,\tau)$:}
Using Minkowski's inequality, the dispersive estimate and H\"older's inequality and the boundedness of potential energy, one has
\begin{align*}
  \|z(t,\tau)\|_{L_x^{p+1}}\lesssim & \int_{t-\tau}^t(t-s)^{-\frac{2(p-1)}{p+1}}\|u\|_{L_x^{p+1}}^p\;ds\\
  \lesssim&\Big(\int_{t-\tau}^t(t-s)^{-\frac{3p-1}{2(p+1)}}\;ds\Big)^\frac{4(p-1)}{3p-1}\Big(\int_{t-\tau}^t\|u\|_{L_x^{p+1}}^{\frac{3p-1}{3-p}p}
  \;ds\Big)^\frac{3-p}{3(p-1)}\\
  \lesssim&\tau^{\frac{2(p-1)(3-p)}{(p+1)(3p-1)}}\Big(\int_{t-\tau}^t\|u\|_{L_x^{p+1}}^{p+1}\;ds\Big)^\frac{3-p}{3(p-1)}\|u\|_{L_t^\infty L_x^{p+1}}^{\frac{4p^2-3p-3}{3(p-1)}}\\
  \leq&L\tau^{\frac{2(p-1)(3-p)}{(p+1)(3p-1)}+\frac{3-p}{3(p-1)}}\Big(\sup_{s>t-\tau}\int_{|x|>s\log s}|u(s,x)|^{p+1}\;dx\Big)^\frac{3-p}{3(p-1)}\\
  &+L\tau^{\frac{2(p-1)(3-p)}{(p+1)(3p-1)}}
  \Big(\int_{t-\tau}^t\int_{|x|\leq s\log s}|u(s,x)|^{p+1}\;dx\;ds\Big)^\frac{3-p}{3(p-1)},
\end{align*}
for some $L>0$.

Therefore,
\begin{align}\nonumber
  \|u(t)\|_{L_x^{p+1}}\leq  & \|v(t)\|_{L_x^{p+1}}+K\tau^{-\frac{p-1}{p+1}}\\\nonumber
  &+L\tau^{\frac{2(p-1)(3-p)}{(p+1)(3p-1)}+\frac{3-p}{3(p-1)}}\Big(\sup_{s>t-\tau}\int_{|x|>s\log s}|u(s,x)|^{p+1}\;dx\Big)^\frac{3-p}{3(p-1)}\\\label{equ:utestex}
  &+L\tau^{\frac{2(p-1)(3-p)}{(p+1)(3p-1)}}
  \Big(\int_{t-\tau}^t\int_{|x|\leq s\log s}|u(s,x)|^{p+1}\;dx\;ds\Big)^\frac{3-p}{3(p-1)}.
\end{align}
For any $\epsilon>0$, there exists $\tau_\epsilon>b$ such that for any $\tau>\tau_\epsilon$ s.t.
\begin{equation}\label{equ:ktauessm}
  K\tau^{-\frac{p-1}{p+1}}\leq K\tau_\epsilon^{-\frac{p-1}{p+1}}=\frac{\epsilon}{4}.
\end{equation}
On the other hand, by Step 1 and \eqref{equ:dispde}, we deduce that there exists $t_1>\tau_\epsilon$ such that for any $t\geq t_1$
\begin{equation}\label{equ:vtlta}
  \|v(t)\|_{L_x^{p+1}}+L\tau_\epsilon^{\frac{2(p-1)(3-p)}{(p+1)(3p-1)}+\frac{3-p}{3(p-1)}}\Big(\sup_{s>t-\tau_\epsilon}\int_{|x|>s\log s}|u(s,x)|^{p+1}\;dx\Big)^\frac{3-p}{3(p-1)}<\frac{\epsilon}{2}.
\end{equation}
  Furthermore, by Step 2, we derive that there exists $t_0>\max\{t_1+b,2\tau_\epsilon\}$ such that
  \begin{equation}\label{equ:lsmscal}
    L\tau_\epsilon^{\frac{2(p-1)(3-p)}{(p+1)(3p-1)}}
  \Big(\int_{t_0-2\tau_\epsilon}^{t_0}\int_{|x|\leq s\log s}|u(s,x)|^{p+1}\;dx\;ds\Big)^\frac{3-p}{3(p-1)}<\frac{\epsilon}{4}.
  \end{equation}
Noting that for $t\in [t_0-b,t_0]$ and $\tau_\epsilon>b,$ we have
$$[t-\tau_\epsilon,t]\subset[t-b,t]\subset [t_0-2\tau_\epsilon,t_0].$$
  Thus,
\begin{equation}\label{equ:smalscales}
  L\tau_\epsilon^{\frac{2(p-1)(3-p)}{(p+1)(3p-1)}}
  \Big(\int_{t-\tau_\epsilon}^t\int_{|x|\leq s\log s}|u(s,x)|^{p+1}\;dx\;ds\Big)^\frac{3-p}{3(p-1)}<\frac{\epsilon}{4}.
\end{equation}
  Plugging \eqref{equ:ktauessm}, \eqref{equ:vtlta} and \eqref{equ:smalscales}  into \eqref{equ:utestex}, we obtain the claim \eqref{equ:t0clsoe}.

{\bf Step 4:} Now we turn to show \eqref{equ:poedecay}. It is equivalent to show that for any $\epsilon>0$, there exists $T_\epsilon>0$ such that for any $t>T_\epsilon$
\begin{equation}\label{equ:podelim}
  \|u(t,x)\|_{L_x^{p+1}}<\epsilon.
\end{equation}
From Step 3, we know that for any $t>\tau>0$
\begin{equation}\label{equ:threstm}
  \|u(t,x)\|_{L_x^{p+1}}\leq \|v(t)\|_{L_x^{p+1}}+K\tau^{-\frac{p-1}{p+1}}+\|z(t,\tau)\|_{L_x^{p+1}}.
\end{equation}
For any $\epsilon>0$, we take $\tau_\epsilon>0$ such that
\begin{equation}\label{equ:tauepscho}
  K\tau_\epsilon^{-\frac{p-1}{p+1}}=\frac{\epsilon}{4}.
\end{equation}
On the other hand, by \eqref{equ:dispde}, we can choose $t_1>\tau_\epsilon>0$ such that
\begin{equation}\label{equ:vteps}
  \|v(t)\|_{L_x^{p+1}}<\frac{\epsilon}{4},\quad \forall\;t>t_1.
\end{equation}
Thus,
\begin{equation}\label{equ:utl4}
  \|u(t,x)\|_{L_x^{p+1}}\leq\frac{\epsilon}{2}+\|z(t,\tau_\epsilon)\|_{L_x^{p+1}},\quad \forall\;t>t_1.
\end{equation}
Using the dispersive estimate, we get
\begin{align}\nonumber
 \|z(t,\tau_\epsilon)\|_{L_x^{p+1}}\lesssim & \int_{t-\tau_\epsilon}^t (t-s)^{-\frac{2(p-1)}{p+1}}\|u(s)\|_{L_x^{p+1}}^p\;ds\\\label{equ:ztest}
 \leq& M\tau_\epsilon^{\frac{3-p}{p+1}}\sup_{s\in[t-\tau_\epsilon,t]}\|u(s)\|_{L_x^{p+1}}^p
\end{align}
for some $M>0$.
Hence,
\begin{equation}\label{equ:utestlp}
   \|u(t,x)\|_{L_x^{p+1}}\leq\frac{\epsilon}{2}+M\tau_\epsilon^{\frac{3-p}{p+1}}\sup_{s\in[t-\tau_\epsilon,t]}\|u(s)\|_{L_x^{p+1}}^p,\quad \forall\;t>t_1.
\end{equation}

On the other hand, applying \eqref{equ:t0clsoe} with $b=\tau_\epsilon$, we derive that there exists $t_0\geq\tau_\epsilon$ such that
\begin{equation}\label{equ:utsmal}
  \sup_{t\in[t_0-\tau_\epsilon,t_0]}\|u(t)\|_{L_x^{p+1}}< \epsilon.
\end{equation}

Now, we utilize the bootstrap argument to prove \eqref{equ:podelim}. To do this, we define
\begin{equation}\label{equ:defteps}
  t_\epsilon:=\sup\{t\geq t_0:\;\|u(s)\|_{L_x^{p+1}}< \epsilon,\quad\forall\;s\in[t_0-\tau_\epsilon,t]\}.
\end{equation}
Then, it is equivalent to show that $t_\epsilon=+\infty.$ By contradiction, we assume that $t_\epsilon<+\infty$. Then,
\begin{equation}\label{equ:tepsd}
  \|u(t_\epsilon)\|_{L_x^{p+1}}=\epsilon.
\end{equation}
Using \eqref{equ:utl4} with $t=t_\epsilon$, we obtain
$$\epsilon\leq \frac{\epsilon}{2}+M\tau_\epsilon^{\frac{3-p}{p+1}}\epsilon^p\Longrightarrow\;\tau_\epsilon^{\frac{3-p}{p+1}}\epsilon^{p-1}\geq\frac{1}{2M}.$$
This together with \eqref{equ:tauepscho} implies that
$$\tau_\epsilon^{\frac{3-p}{p+1}}(4K\tau_\epsilon^{-\frac{p-1}{p+1}})^{p-1}\geq\frac{1}{2M}\Longrightarrow\;\tau_\epsilon^{\frac{p^2-p-2}{p+1}}
\leq 2M(4K)^{p-1}$$
 which contradicts with the fact that
$$\lim_{\epsilon\to0}\tau_\epsilon=+\infty.$$
Therefore, we conclude the proof of Proposition \ref{prop:pedecay}.

\end{proof}

\subsection{Global space-time bound}

In this subsection, we utilize the decay of potential energy in Proposition \ref{prop:pedecay} and  the continuous argument to show the global space-time bound for the global solution.

\begin{proposition}\label{prop:glostb}
Let $u:\;\R\times\R^4\to \C$ be the global solution to \eqref{equ:nlsab4}. Then, there holds
\begin{equation}\label{equ:globaspse}
  \| u(t,x)\|_{L_t^qH^{1,r}_{\A}(\R\times\R^4)}<+\infty,
\end{equation}
for any $(q,r)\in\Lambda_0^S$. As a consequence, the solution $u(t,x)$ scatters.
\end{proposition}

\begin{proof}
We first utilize Proposition \ref{prop:pedecay} and continuous argument to prove
\begin{equation}\label{equ:l4ewfs}
  \|u\|_{L_t^\gamma([0,\infty), H_{\bf A}^{1,\rho}(\R^4))}<+\infty,\;(\gamma,\rho)=\big(\tfrac{p+1}{p-1},p+1\big).
\end{equation}
Using local well-posed theory, we get for any $T>0$,
$$\|u\|_{L_t^\gamma([0,T], H_{\A}^{1,\rho}(\R^4))}\leq C\big(T,M(u_0),E(u_0)\big).$$
Thus, we only need to show for some $T>0$
\begin{equation}\label{equ:ulthsa}
  \|u\|_{L_t^\gamma([T,\infty), H_{\A}^{1,\rho}(\R^4))}<+\infty.
\end{equation}
For $t>T>0$, we have by Duhamel's formula
$$u(t)=e^{i(t-T)\LL_{\A}}u(T)-i\int_T^t e^{i(t-s)\LL_{\A}}(|u|^{p-1}u)(s)\;ds.$$
By Strichartz's estimate and Sobolev embedding, we get
\begin{align*}
  \|u\|_{L_t^\gamma([T,\infty), H_{\A}^{1,\rho}(\R^4))}\leq& C\|u\|_{H^1_{\A}}+C\big\||u|^{p-1}u\big\|_{L_t^{\gamma'}([T,\infty),H_{\A}^{1,\rho'})}\\
  \leq& C\|u\|_{H^1_{\A}}+C\|u\|_{L_t^\gamma([T,\infty), H_{\A}^{1,\rho}(\R^3))}\|u\|_{L_t^\gamma L_x^\rho}^{\frac{3-p}{p-1}}\sup_{t\geq T}\|u(t)\|_{L_x^{p+1}}^{\frac{p^2-p-2}{p-1}}\\
  \leq& C\|u\|_{H^1_{\A}}+C\sup_{t\geq T}\|u(t)\|_{L_x^{p+1}}^{\frac{p^2-p-2}{p-1}}\|u\|_{L_t^\gamma([T,\infty), H_{\A}^{1,\rho}(\R^4))}^\frac2{p-1}\\
  \triangleq&C\|u\|_{H^1_{\A}}+C\epsilon(T)\|u\|_{L_t^\gamma([T,\infty), H_{\A}^{1,\rho}(\R^4))}^\frac2{p-1}.
\end{align*}
From Proposition \ref{prop:pedecay}, we know that
$$\lim_{T\to\infty}\epsilon(T)=0.$$
This together with the continuous argument yields \eqref{equ:ulthsa}. And so \eqref{equ:l4ewfs} follows.

Furthermore, using  the estimate
$$\big\||u|^{p-1}u\big\|_{L_t^{\gamma'} H_{\A}^{1,\rho'}}\lesssim\|u\|_{L_t^\infty H_{\A}^1}^{\frac{p^2-p-2}{p-1}}\|u\|_{L_t^\gamma H_{\A}^{1,\rho}}^\frac2{p-1},$$ and Strichartz estimate,
one can now deduce that $u\in L_t^qH^{1,r}_{\A}$ for any admissible $(q,r)\in\Lambda_0^S$.

\end{proof}

{\bf Acknowledgements} The authors would like to express their great gratitude to Z. Zhao for his helpful
discussions.

\begin{center}

\end{center}

\begin{thebibliography}{99}

%



%
%
%
%
%

\bibitem{AB}
{ Y. Aharonov, and D. Bohm}, Significance of electromagnetic potentials in quantum theory,
Phys. Rev. Lett., 115(1959), 485--491.

\bibitem{AW} M. Alford and F. Wilczek, Aharonov-Bohm interaction of cosmic strings with matter, Phys. Rev. Lett. 62(1989), 1071.


\bibitem{AFVV}
P. D'Ancona, L. Fanelli, L. Vega, and N. Visciglia, Endpoint Strichartz estimates for the magnetic Schr\"odinger equation,
Comm. Math. Phys., 324(2013), 1033-1067.


\bibitem{BPSS}
 N. Burq, F. Planchon, J. Stalker, and A. S.
Tahvildar-Zadeh, Strichartz estimates for the wave and Schr\"odinger
equations with the inverse-square potential, J. Funct. Anal., 203
(2003), 519-549.


\bibitem{BVZ} J. Barcel\'o, L. Vega and M. Zubeldia, The forward problem for the electromagnetic Helmholtz
equation with critical singularities, Advances in Math., 240(2013), 636-671.


\bibitem{CYZ} F. Cacciafesta, Z. Yin and J. Zhang, Generalized Strichartz estimates for wave and Dirac equations in Aharonov-Bohm magnetic fields, arXiv:2008.00340.


%
%
%
%




\bibitem{CCL}  J. Colliander, M. Czubak, Magdalena and J. Lee,  Interaction Morawetz estimate for the magnetic Schr\"odinger equation and applications. Adv. Differential Equations, 19(2014), 805-832.



\bibitem{CKSTT} J. Colliander, M. Keel, G. Staffilani, H. Takaoka, and T. Tao. Global existence and scattering for rough solutions of a nonlinear Schr\"odinger equation on $\R^3$. Comm. Pure Appl. Math., 57(2004), 987-1014.


\bibitem{CKSTT07}
J. Colliander, M. Keel, G. Staffilani, H. Takaoka, and T. Tao,
Global well-posedness and scattering for the energy-critical
nonlinear
    Schr\"odinger equation in $\R^3$.
Annals of Math., 167 (2008), 767-865.

\bibitem{CGMZ} J. Chong, M. Grillakis, M. Machedon and Z. Zhao, Global estimates for the Hartree-Fock-Bogoliubov equations, Commun. PDE., 46(2021), 2015-2055.

%
%
%







%
%
%
%

%
%




\bibitem{FFFP1}
 L. Fanelli, V. Felli, M. A. Fontelos, and A. Primo, Time decay of scaling critical electromagnetic Schr\"odinger flows,
J. Funct. Anal., 258(2010), 3227-3240.

\bibitem{FFFP}
 L. Fanelli, V. Felli, M. A. Fontelos, and A. Primo, Time decay of scaling invariant electromagnetic Schr\"odinger equations on the plane,
Comm. Math. Phys., 337(2015), 1515-1533.



\bibitem{FKLV}
 L.  Fanelli, D. Krej$\check{c}$i$\check{r}$\'ik, A. Laptev, L. Vega,  On the improvement of the Hardy inequality due to singular magnetic fields,
  Commun. PDE, 45(2020), 1202-1212.



%


%

\bibitem{FV}
 L. Fanelli and L. Vega, Magnetic virial identities, weak dispersion and Strichartz inequalities, Math. Ann., 344(2009), 249-278.



\bibitem{FZZ} L. Fanelli, J. Zhang and J. Zheng,
\newblock Dispersive estimates for 2D-wave equations with critical potentials,
\newblock arXiv:2003.10356v3.

\bibitem{Fu} D. Fujiwara, Remarks on convergence of the Feynman path integrals, Duke Math. J., 47 (1980), 559-600.

\bibitem{GV} J. Ginibre and G. Velo, The global Cauchy problem for the nonlinear Schr\"odinger equation revisited, Ann. Inst. H. Poincar\'e Anal. Non Lin\'eaire, 2(1985), 309-327.


\bibitem{GYZZ} X. Gao, Z. Yin, J. Zhang and J. Zheng, Decay and Strichartz estimates in critical electromagnetic fields, arXiv:2003.03086.


\bibitem{H} Y. Hong, Strichartz estimates for $N$-body Schr\"odinger operators with small potential interactions, DCDS 37(2017), 5355-5365.



%




%
%




\bibitem{HHLT} M. Hoffmann-Ostenhof, T. Hoffmann-Ostenhof, A. Laptev and J. Tidblom, Many particle Hardy inequalities, J. London Math.
Soc. 77(2008), 99-115.

\bibitem{HZ} A. Hassell and J. Zhang, Global-in-time Strichartz estimates on nontrapping asymptotically conic manifolds,
Analysis \& PDE, 9(2016), 151-192.

\bibitem{J89} J. K. Jain, Composite-fermion approach for the fractional quantum Hall effect, Phys. Rev. Lett. 63(1989), 199-202.

\bibitem{KMVZZ} R. Killip, C. Miao, M. Visan, J. Zhang, and J. Zheng,
Sobolev spaces adapted to the Schr\"odinger operator with inverse-square potential, Math. Z., 288 (2018), 1273-1298.

\bibitem{KT}
 M. Keel and T. Tao, Endpoint Strichartz estimates,
Amer. J. Math., 120(1998), 955-980.

\bibitem{L83} R. B. Laughlin, Anomalous quantum Hall effect: an incompressible quantum fluid with fractionally chaged excitations, Phys. Rev. Lett. 50(1983), 1395-1398.

%
%
%
%
%

%
%
 \bibitem{PT89}
{M. Peshkin, and A. Tonomura}. The Aharonov-Bohm Effect. Lect. Notes Phys., 340(1989).


\bibitem{sogge} C. D. Sogge, Fourier Integrals in Classical Analysis, Cambridge Tracts in Mathematics,
vol. 105, Cambridge University Press, Cambridge, 1993.



%
%
%
%


%

%
%
%
%
%
%
%
%
%
%
%

%



%
%
%

\bibitem{RS}  I. Rodnianski and W. Schlag, Time decay for solutions of Schr\"odinger equations with rough and time-dependent potentials, Invent. Math., 155(2004), 451-513.
%

%
%

\bibitem{S}  W. Schlag, Dispersive estimates for Schr\"odinger operators: a survey. Mathematical aspects of nonlinear dispersive equations, 255-285, Ann. of Math. Stud. 163, Princeton Univ. Press, Princeton, NJ, 2007.


%
%


%
%
%
%

\bibitem{TVZ} T. Tao, M. Visan, X. Zhang, The nonlinear Schr\"odinger equation with combined power-type nonlineari-
ties, Commun. PDE, 32(2007), 1281–1343.

\bibitem{Taylor}
M. Taylor, Partial Differential Equations, Vol II,
Springer, 1996.

\bibitem{Taylor:Tools} M. E. Taylor, Tools for PDE. Mathematical Surveys and Monographs, 81. American Mathematical Society, Providence, RI, 2000.

\bibitem{Watson}  G. N. Watson, A Treatise on the Theory of Bessel Functions. Second Edition, Cambridge
University Press, 1944.

\bibitem{ZZ18}  J. Zhang, Resolvent and spectral measure for Schr\"odinger operators on flat Euclidean cones, arXiv:2010.12838.








%
%













%





%


%
%













\end{thebibliography}
\end{document}